\documentclass[final]{amsart}
\usepackage{amssymb,amsmath, color}
\usepackage{comment}

\usepackage[notref,notcite]{showkeys}

\newtheorem{Theorem}{Theorem}[section]
\newtheorem{lemma}[Theorem]{Lemma}

\theoremstyle{definition}
\newtheorem{remark}[Theorem]{Remark}

\DeclareMathOperator*\prox{prox}
\newcommand{\R}{{\mathbb R}}
\newcommand{\N}{{\mathbb N}}
\DeclareMathOperator{\argmin}{argmin}
\newcommand{\interior}{{\rm int}\kern 0.06em}
\newcommand{\inte}{{\rm int}\kern 0.06em}
\newcommand{\cl}{{\rm cl}\kern 0.06em}
\newcommand{\zer}{{\rm zer}\kern 0.06em}
\newcommand{\gph}{{\rm gph}\kern 0.06em}
\newcommand{\dom}{{\rm dom}\kern 0.06em}
\newcommand{\pr}{{\rm pr}\kern 0.06em}
\newcommand{\e}{\varepsilon}

\newcommand{\p}{{\partial}}
\newcommand{\To}{\longrightarrow}
\def\a{\alpha}

\def\ox{\overline{x}}

\def\b{\beta}
\def\d{\delta}

\def\g{\gamma}
\def\m{\mu}
\def\s{\sigma}
\def\l{\lambda}
\def\<{\langle}
\def\>{\rangle}

\newcommand{\cH}{{\mathcal H}}

\newcommand{\ds}{\displaystyle}

\usepackage{epsfig}
\usepackage{geometry}
 \usepackage{pict2e}

 \if
 {
 \usepackage[pageref]{backref}
\renewcommand*{\backrefalt}[4]{%
\ifcase #1 %
(Not cited)%
\or
(Cited on p.~#2)%
\else
(Cited on pp.~#2)%
\fi
}

}
\fi
\title{On the convergence of an inertial proximal algorithm with a Tikhonov regularization term }

\author{Szil\'ard Csaba L\'aszl\'o  }
\thanks{Corresponding Author: Szil\'ard Csaba L\'aszl\'o}
\thanks{Affiliation: Technical University of Cluj-Napoca, Department of Mathematics, Str. Memorandumului nr. 28, 400114 Cluj-Napoca, Romania}
\thanks{ e-mail: szilard.laszlo@math.utcluj.ro}
\thanks{This work was supported by a grant of the Ministry of Research, Innovation and Digitization, CNCS -
UEFISCDI, project number PN-III-P1-1.1-TE-2021-0138, within PNCDI III}

\begin{document}

\begin{abstract}
This paper deals with an inertial proximal algorithm that contains a Tikhonov regularization term, in connection  to the minimization problem of a convex lower semicontinuous function $f$.  We show that for appropriate Tikhonov regularization parameters the  value of the objective function in the sequences  generated by our algorithm converges fast (with arbitrary rate) to the global minimum of the objective function and the  generated sequences converges weakly  to a minimizer of the objective function. We also  obtain the fast convergence of subgradients and the discrete velocities towards zero and some sum estimates. Further, we obtain strong convergence results for the generated sequences and also fast convergence for the function values  and discrete velocities for the same constellation of the parameters involved. Our analysis reveals that the extrapolation coefficient, the stepsize and the Tikhonov regularization coefficient are strongly correlated and there is a critical setting of the parameters that separates the cases when strong convergence results or weak convergence results can be obtained.
\end{abstract}
\maketitle

\noindent \textbf{Key Words.}  convex optimization,  inertial proximal algorithm, Tikhonov regularization, strong convergence, convergence rate \vspace{1ex}

\noindent \textbf{AMS subject classification.}  46N10, 65K05,  65K10, 90C25, 90C30

\markboth{S.C. L\'ASZL\'O}{PROXIMAL ALGORITHM WITH TIKHONOV REGULARIZATION}

\section{Introduction}
\subsection{The state of the art in continuous case} The strong convergence of the trajectories  of  second order continuous dynamical systems with a Tikhonov regularization term to a minimizer of minimum norm of a smooth convex objective function were the subject of many recent investigations, (see \cite{AL-siopt,ABCR,ACR,ACR2,AL-nemkoz,BCL,BCLstr,L-jde,L-amop}). These dynamical systems lead via explicit/implicit discretizations  to inertial algorithms with a Tikhonov regularization term therefore the asymptotical behaviour of the generated trajectories give an insight into the behaviour of the sequences generated by the inertial algorithms obtained via discretization from these dynamical systems. However, until recently it was thought that for a setting of the parameters involved one can obtain fast convergence of the function values in a generated trajectory to the minimum of the objective function and (eventually) weak convergence of the trajectory to a minimizer of the objective function, meanwhile for another setting of the parameters one can obtain strong convergence results of the generated trajectories only (without fast rates for the function values). Even more the strong convergence results obtained were  in the form $\liminf_{t\to+\infty}\|x(t)-x^*\|=0$ where $x$ is a trajectory of the dynamical system and $x^*$ is the minimal norm minimizer of the objective function. Indeed, in \cite{ACR}  the authors associated to the optimization problem $\min_{x\in\mathcal{H}} f(x)$  the  dynamical system
\begin{align}\label{DynSysACR}
&\ddot{x}(t) + \frac{\a}{t} \dot{x}(t) +\nabla f\left(x(t)\right) +\e(t)x(t)=0,\,
x(t_0) = u_0, \,
\dot{x}(t_0) = v_0,
\end{align}
where $t\ge t_0 > 0,\a\ge 3$, $(u_0,v_0) \in \mathcal{H} \times \mathcal{H}$ and the Tikhonov regularization parameter $\e(t)$ is a nonincreasing positive function satisfying $\lim_{t\to+\infty}\e(t)=0.$ Here $\mathcal{H}$ is a real Hilbert space and the objective function $f:\mathcal{H}\to\R$ is  smooth and convex. Now, if the Tikhonov regularization parameter is $\epsilon (t) = \frac{c}{t^p},\,c>0$,   $ p>0 $ and $\alpha > 3$ then according to \cite{ACR} the following statements hold.
When  $ p>2 $ one has $ f\left(x(t)\right)-\min f = o (t^{-2})$, $\|\dot{x}(t)\|= o (t^{-1})$ as $t\to+\infty$
and  $x(t)$ converges weakly to a minimizer of $f$. Further,  in the case   $1 < p<2$ or $p=2$ and $c > \frac{2}{9}\alpha(\alpha -3)$
one has
  $\liminf_{t\rightarrow\infty} \Vert x(t)-x^{*}\Vert =0,$
where $x^{*}$ is the element of minimum norm of $\argmin f$.

One can observe that the case $p=2$ is critical in the sense that separates the case when fast convergence of the function values and weak convergence of the trajectories are obtained and when strong convergence results for the trajectories hold. Some similar results were obtained in \cite{BCL} for a second order dynamical system with Hessian driven damping.

The first breakthrough was made in \cite{AL-nemkoz} where the authors succeeded to obtain both fast convergence towards the minimal value of $f$ and the strong convergence of the generated trajectories towards the element of minimum norm of $\argmin f$. More precisely, in \cite{AL-nemkoz} it is shown that if $\e(t)=\frac{c}{t^2}$ and $\a>3$ in \eqref{DynSysACR} then $ f\left(x(t)\right)-\min f = \mathcal{O} (t^{-2})$, $\|\dot{x}(t)\|= \mathcal{O} (t^{-1})$ as $t\to+\infty$, and there is strong convergence to the minimum norm solution $\liminf_{t\rightarrow\infty} \Vert x(t)-x^{*}\Vert =0$. A similar result has been obtained in \cite{AL-siopt} for a second order dynamical system with implicit Hessian driven damping.
These results have been improved in \cite{ABCR} where the authors studied the dynamical system introduced in \cite{AL-nemkoz}, that is
\begin{align}\label{DynSysABCR}
\ddot{x}(t) + \d\sqrt{\e(t)} \dot{x}(t) +\nabla f\left(x(t)\right) +\e(t)x(t)=0,\,x(t_0) = u_0, \,
\dot{x}(t_0) = v_0,\,t\ge t_0>0,\,\d>0
\end{align}
in connection to the minimization problem  with a smooth convex objective function $f$. They have shown that in case the Tikhonov regularization parameter has the form $\e(t)=\frac{1}{t^p}$ with  $p<2$ then $ f\left(x(t)\right)-\min f = \mathcal{O} (t^{-p})$, $\|\dot{x}(t)\|= \mathcal{O} (t^{-\frac{p+2}{4}})$ as $t\to+\infty$, and there  is 'full' strong convergence to the minimum norm solution, that is, $\lim_{t\rightarrow\infty} \Vert x(t)-x^{*}\Vert =0$. Some similar results have been obtained in \cite{ABCRamop} for a dynamical system with a Hessian driven damping.

 The results from \cite{ABCR} have  further been extended and improved in \cite{L-jde}. Indeed, \cite{L-jde}   deals with  the dynamical system
\begin{align}\label{DynSys}
&\ddot{x}(t) + \frac{\a}{t^q} \dot{x}(t) +\nabla f\left(x(t)\right) +\frac{c}{t^p}x(t)=0,\,
x(t_0) = u_0, \,
\dot{x}(t_0) = v_0,
\end{align}
where $t_0 > 0$, $0<q<1$ and $(u_0,v_0) \in \mathcal{H} \times \mathcal{H}$. Note that for the Tikhonov regularization parameter $\e(t)=\frac{c}{t^p}$ the dynamical system \eqref{DynSysABCR} is a special case of \eqref{DynSys} obtained for $q=\frac{p}{2}.$ However, according to \cite{L-jde}, $p=2q$ is not the best choice since, for instance, for a fixed $q$ one may have the convergence rate
$f(x(t))-\min f=O\left(t^{-\frac{4q+2}{3}}\right)$, which is obviously better than the rate $ f\left(x(t)\right)-\min f = \mathcal{O} (t^{-2q})$ obtained in \cite{ABCR}. Further, the choice of the inertial parameter and Tikhonov regularization parameter in  \eqref{DynSys} allow to make a comprehensive study on how these parameters are correlated. More precisely, for $0<p<q+1$ one obtains strong convergence of the trajectories to the minimal norm solution and fast convergence rates for the decay $f(x(t))-\min f$, for $q+1<p\le 2$ the trajectories converge weakly to a minimizer of $f$ and the fast convergence rates for the decay $f(x(t))-\min f$ is provided. The case $p=q+1$ is critical in the sense that separates the cases when weak and strong convergence of the trajectories can be obtained, however also in this case fast convergence rates for the decay $f(x(t))-\min f$ hold.

\subsection{The problem formulation, motivation and a model result}

As we have seen the study of second order continuous dynamical systems with a Tikhonov regularization term in connection to the minimization problem of a smooth convex function has a rich literature. In contrast in the discrete case, the case of inertial algorithms with a Tikhonov regularization term,  there are no  results similar to those emphasized before. The aim of this paper to reduce the gap between the continuous and discrete case.
 To this purpose we introduce a proximal inertial algorithm which, for constant stepsize, can be seen as an implicit discretization of the dynamical system \eqref{DynSys}. However, in our algorithm we do not assume that the objective function is smooth and we consider a variable stepsize parameter. As it was expected, the most important features of a trajectory generated by the dynamical system \eqref{DynSys} are inherited by the sequences generated by our algorithm.  This underlines again the importance of the study of the continuous case (see \cite{ACR1,CPS}), when one ought to design an optimization algorithm  with desirable properties.
 
 Consider the minimization problem
$$\mbox{(P)}\,\,\,\inf_{x\in\cH}f(x),$$
where $\mathcal{H}$ be a Hilbert space endowed with the scalar product $\< \cdot,\cdot\>$ and norm $\|\cdot\|$ and
$f:\mathcal{H}\To\overline{\R}=\R\cup\{+\infty\}$ is  a convex proper lower semicontinuous  function whose solution set $\argmin f$ is  nonempty.
We associate to (P) the following inertial proximal algorithm: for all $k\ge 1$
\begin{align}\label{algo}\tag{PIATR}
\begin{cases}
x_0,x_{1}\in\mathcal{H}\\
y_k= x_k+\a_k(x_k -  x_{k-1})\\
x_{k+1}={ \rm prox}_{\l_k f}\left( y_k - c_k x_k\right),
\end{cases}
\end{align}
where  $\a_k=1-\frac{\a}{k^q},\,\a>0,\,0<q\le 1$ is the inertial parameter, $c_k=\frac{c}{k^p},\,c,\,p>0$ is the Tikhonov regularization parameter and we assume that the stepsize has the form $\l_k=\l k^\d,\,\l>0,\,\d\in\R,$ for all $k\ge 1$. Further, $\prox\nolimits_{s  f} : {\cH} \rightarrow {\cH}, \quad \prox\nolimits_{s f}(x)=\argmin_{y\in {\cH}}\left(f(y)+\frac{1}{2s}\|y-x\|^2\right),$ denotes the proximal point operator of the convex function $s  f$. The name of the algorithm stands for Proximal Inertial Algorithm with Tikhonov Regularization. 

The forms of the extrapolation parameter $\a_k,$ Tikhonov regularization parameter $c_k$ and the stepsize $\l_k$ are motivated by the fact that for these forms one can easily see for which constellation of the parameters the weak convergence and the strong convergence of the  sequences generated by algorithm \eqref{algo} can be obtained. Indeed, our analysis reveals that the inertial coefficient $\a_k$ the stepsize $\l_k$ and the Tikhonov regularization coefficient $c_k$ are strongly correlated, there is a setting of the parameters when weak convergence of the sequences generated by \eqref{algo} to a minimizer of the objective function $f$ can be shown and also fast convergence rates for the function values in the generated sequences to the global minimum of the objective function and fast convergence of the discrete velocity to zero can be obtained. For another constellation of the parameters involved one can obtain strong convergence results of the generated sequences to the minimum norm minimizer of $f$, but also rapid convergence rates for  the function values and discrete velocity. Further, there is a setting of the parameters that separates the case when weak convergence  and the case when strong convergence of the sequences generated by algorithm \eqref{algo} can be obtained.
We emphasize that  the form of inertial parameter $\a_k$ was inspired by the famous FISTA method \cite{BT}, where the inertial parameter has the form $1-\frac{\a}{k},$ see also \cite{AAD,AAD1,Nest1}. Even more, if we take $q=1$ in our inertial parameter, then algorithm \eqref{algo} can be seen as a perturbed version of the inertial proximal algorithm (IPA) studied in \cite{ACR1} (see also \cite{ABotCest,ACPR,APmathpr}). Nevertheless, since $c_k$ is a positive, nonincreasing sequence that goes to $0$ as $k\to+\infty$, the perturbation term  $c_k x_k$  in algorithm \eqref{algo} is actually a Tikhonov regularization term, which may assure the strong convergence of a generated sequence to the minimizer of minimal norm of the objective function $f$. For a better insight on Tikhonov regularization techniques we refer to \cite{AL-siopt,abc2,ACR,att-com1996,AC,AL-nemkoz,BCL,BGMS,CPS,JM-Tikh,L-jde,Tikh,TA}. As it was expected, in case $q=1$ for some settings of the Tikhonov regularization parameter $c_k$ we reobtain the results from \cite{ACR1} (see Theorem \ref{convergencealgorithm}), however according to Remark \ref{comACR1}, just as in continuous case, the best choice of $q$ in the inertial parameter $\a_k$ is not $q=1$, but rather $0<q<1,$ since in this case improved rates can be obtained.
 
 Though the explicit form of the proximal point operator is sometimes hard to be computed the proximal point algorithms  are the basic tools for solving nonsmooth convex optimization problems. This is due to the fact that when one deals with an optimization problem having in its objective a convex lower semicontinuous function with extended real values, then the objective function is not differentiable and therefore the simple gradient type methods (see for instance \cite{Nest1}) are not allowed. Of course by rewriting $\prox\nolimits_{s f}$ as the resolvent operator of the subdifferential of the convex function $sf$, that is $\prox\nolimits_{s f}(x)=(I+s\p f)^{-1}(x),$ algorithm \eqref{algo} can be reformulated as the subdifferential inclusion
\begin{equation}\label{formx}\tag{SDI}
x_{k+1}\in\a_k(x_{k}-x_{k-1})-\l_k \p f(x_{k+1}) +\left(1- c_k\right) x_k,\mbox{ for all }k\ge 1,
\end{equation}
but this formulation  is not suitable for implementation since the subdifferential of the objective function, namely $\p f$, usually cannot be computed. Consequently, the single general tool that one can use in this instance is the proximal point operator, and this is underlined by the fact that a rich literature has been devoted to proximal-based inertial algorithms  \cite{ACR1,APR,BT,BCL1,CD,JM,LP,MM,MO}. We are aware that our general results obtained in this paper, when we consider variable stepsize in our algorithm, are important mainly by a theoretical point of view, however, for constant stepsize only one proximal point operator must be computed and in this case our algorithm is also suitable for numerical implementation. Indeed, if we fix the stepsize $\l_k\equiv 1$ in our algorithm, then the main contributions of the paper to the state of the art can be summarized in the following result, see Theorem \ref{convergencealgorithm}, Theorem \ref{weakconvergencealgorithm} and Theorem \ref{strconvergencelambda1}.

\begin{Theorem}\label{tmodel} Assume that $0<q<1,\,\a,c,p>0$ and for some starting points $x_0,x_1\in\mathcal{H}$ let $(x_k)$ be a sequence generated by \eqref{algo}, that is,
$x_{k+1}=\prox\nolimits_{f}\left(x_k+\left(1-\frac{\a}{k^q}\right)(x_k -  x_{k-1})-\frac{c}{k^p}\right),\mbox{ for all }k\ge 1.$
For every $k\ge 2$ let us denote $u_k$ the element  from $\p f(x_k)$ that satisfies \eqref{formx} with equality, that is,
$x_{k}=\left(1-\frac{\a}{(k-1)^q}\right)(x_{k-1}-x_{k-2})-u_k +\left(1- \frac{c}{(k-1)^p}\right) x_{k-1}.$
\begin{enumerate}
\item[(i)] If  $q+1<p\le 2$ and for $p=2$ one has $c>q(1-q)$,
then  $(x_k)$ converges weakly to a minimizer of $f.$ Further,
$f(x_k)-\min_{\cH} f=\mathcal{O}(k^{-q-1}),\,\|x_{k}-x_{k-1}\|=\mathcal{O}(k^{-\frac{q+1}{2}})\mbox{ and }\|u_k\|=o(k^{-\frac{q+1}{2}})\mbox{ as } k\to+\infty.$ Moreover,
$\sum_{k=1}^{+\infty} k^{q}(f(x_k)-\min_{\cH} f)<+\infty,$ $\sum_{k=1}^{+\infty} k\|x_k-x_{k-1}\|^2<+\infty$ and $\sum_{k=2}^{+\infty} k^{q+1}\|u_k\|^2<+\infty.$
\item[(ii)] If $p=q+1$ then for all $s\in\left]\frac12,\frac{q+1}{2}\right[$ one has
$f(x_k)-\min_{\cH} f=o(k^{-2s}),$ $\|x_{k}-x_{k-1}\|=o(k^{-s})$ and $\|u_k\|=o(k^{-s})\mbox{ as } k\to+\infty.$
Further,
 $\sum_{k=1}^{+\infty} k^{2s-1}(f(x_k)-\min_{\cH} f)<+\infty,$ $\sum_{k=1}^{+\infty} k^{2s-q}\|x_k-x_{k-1}\|^2<+\infty$ and $\sum_{k=2}^{+\infty} k^{2s}\|u_k\|^2<+\infty.$
\item[(iii)] If $1<p< q+1$ then $\lim_{k\to+\infty}\|x_k-x^*\|=0$, where $x^*$ is the minimal norm element from $\argmin f$.

    Further, if $1<p\le 2q$ then $\|x_k-x_{k-1}\|^2,\,\|u_k\|^2\in \mathcal{O}(k^{q-p-1})\mbox{ as }k\to+\infty$ and $f(x_{k})-\min_{\cH}f= \mathcal{O}(k^{-p})\mbox{ as }k\to+\infty.$

    If $2q< p\le \frac{3q+1}{2}$ then
$\|x_k-x_{k-1}\|^2,\,\|u_k\|^2\in \mathcal{O}(k^{-q-1})\mbox{ as }k\to+\infty$ and $f(x_{k})-\min_{\cH}f= \mathcal{O}(k^{-p})\mbox{ as }k\to+\infty.$

 If $\frac{3q+1}{2}<p<q+1$, then
 $\|x_k-x_{k-1}\|^2,\,\|u_k\|^2\in \mathcal{O}(k^{2p-4q-2})\mbox{ as }k\to+\infty.$
Additionally, if  $\frac{3q+1}{2}<p<\frac{4q+2}{3}$, then $f(x_{k})-\min_{\cH}f= \mathcal{O}(k^{-p})\mbox{ as }k\to+\infty$
and if $\frac{4q+2}{3}\le p<q+1$, then $f(x_{k})-\min_{\cH}f= \mathcal{O}(k^{2p-4q-2})\mbox{ as }k\to+\infty.$

Moreover, if $2q< p$ then
 $\sum_{k=1}^{+\infty} k^{2q}\|u_k\|^2<+\infty$ and
$\sum_{k=1}^{+\infty} k^q\|x_{k+1}-x_{k}\|^2<+\infty.$
\end{enumerate}
\end{Theorem}

Note that the rates presented at (i) in Theorem \ref{tmodel} are in concordance with the results obtained in \cite{AAD} (see also \cite{AAD1}), but we additionally obtain convergence rates  for the subgradient and also some sum estimates. Observe that the sum estimate involving the discrete velocity $\|x_k-x_{k-1}\|$ does not depend by $q.$ Of course $q$ can be arbitrary closed to 1 and therefore the rates obtained at (i) are comparable to Nesterov's rate for the case $\a>3$, see \cite{AAD,AAD1,Nest1}, however our results are valid for every $\a>0.$ Nevertheless, our main scope was to obtain, for the same constellation of the parameters in algorithm \eqref{algo}, fast convergence of the objective function values in the generated sequences to the minimum of the objective function $f$, fast convergence of the discrete velocity $\|x_k-x_{k-1}\|$ to zero and strong convergence of the generated sequences to a minimizer of $f.$ Even more, our aim was to control to which minimizer the generated sequences converge, more precisely our target was to find the minimal norm minimizer of the objective function $f$ and this is the motivation of using Tikhonov regularization in algorithm \eqref{algo}. According to the results presented at (iii) in Theorem \ref{tmodel} the goal emphasized above was fully attained. Though we fix the stepsize to 1 in Theorem \ref{tmodel}, already in this particular case of \eqref{algo}  one can see that indeed the Tikhonov regularization parameter and the inertial parameter are strongly correlated and according to (ii) there is a setting of the parameters that separates the cases when weak convergence and strong convergence can be obtained. We emphasize that due to our knowledge the result presented at (iii) in Theorem \ref{tmodel} is the first strong convergence result in the literature concerning proximal inertial algorithms, though in \cite{AL-nemkoz} the partial strong convergence result $\liminf_{k\to+\infty}\|x_k-x^*\|=0$,  was obtained for the case $q=1$ and $p=2$. Nevertheless, as we mentioned before, $q=1$ is not an optimal choice for our algorithm since in case $q<1$ improved convergence rates can be obtained, further $\liminf_{k\to+\infty}\|x_k-x^*\|=0$ assures only that the sequence $(x_k)$ has a subsequence that converges in the strong topology to $x^*$, meanwhile according to Theorem \ref{tmodel} in this paper we obtain full strong convergence, that is, $\lim_{k\to+\infty}\|x_k-x^*\|=0.$

However, in order to make  a comprehensive analysis, in our algorithm we allow also variable stepsize of the form $\l_k=\l k^\d,\l>0,\,\d\in\R.$
Note that by considering variable stepsize in \eqref{algo} there is a setting of the parameters when one can obtain arbitrary fast rates for the potential energy $f(x_k)-\min f$ or discrete velocity $\|x_k-x_{k-1}\|$, see Theorem \ref{convergencealgorithm}. This result is  in concordance  with the results  obtained by G\"uler in \cite{Gu2}, (see also \cite{Gu1}), however our parameters have a much simpler form.
  According to Theorem \ref{convergencealgorithm}, Theorem \ref{weakconvergencealgorithm} and Theorem \ref{strconvergencelambda1} the stepsize is strongly correlated to the inertial parameter and Tikhonov regularization parameter.

 More precisely, if $q+1<p$ and $\d\ge 0$, then the sequence $(x_k)$ generated by algorithm \eqref{algo} converges weakly to a minimizer of our objective function $f$, see Theorem \ref{weakconvergencealgorithm}. Further, the fast convergence of arbitrary rate of discrete velocity $\|x_k-x_{k-1}\|$ to zero and  convergence of arbitrary rate of the potential energy $f(x_k)-\min_{\cH}f$ to zero  is assured, see Theorem \ref{convergencealgorithm}.  If $1<p<q+1,\,0<\l<1$ and $\d\le 0$, then the strong convergence result $\liminf_{k\to+\infty}\|x_k-x^*\|=0$, where $x^*$ is the minimum norm minimizer of the objective function $f$, is obtained, see Theorem \ref{strconvergencealgorithm}. According to Theorem \ref{strconvergencerates} also in this case the fast convergence of the potential energy $f(x_k)-\min_{\cH}f$ and discrete velocity $\|x_k-x_{k-1}\|$ to zero are assured. As we mentioned before,  similar results were obtained only in \cite{AL-nemkoz}, for the case $q=1$ and $p=2$, which is not covered by our analysis.
We emphasize again  the greatest strength of our paper, that is, for the case $0<q<1,\,1<p<q+1$ and $\l_k\equiv 1$ we are able to obtain 'full' strong convergence to the minimal norm solution $x^*$, that is, $\lim_{k\to+\infty}\|x_k-x^*\|=0$ and we obtain fast convergence of the potential energy $f(x_k)-\min_{\cH}f$ and discrete velocity $\|x_k-x_{k-1}\|$ to zero and even some sum estimates. In order to obtain these results some new techniques have been developed. In case $p=q+1$ neither weak convergence nor strong convergence of the generated sequences can be obtained, but surprisingly, in this case fast convergence of the potential energy  and discrete velocity to zero can be obtained both for the case $\d\ge 0$ and $\d<0.$

\subsection{The organization of the paper}
The paper is organized as follows. In the next section we treat the case $q+1\le p$ in order to obtain fast convergence rates for the function values in the sequence generated by algorithm \eqref{algo} but also for the discrete velocity and subgradient. Further, if $q+1<p$ then the weak convergence of the generated sequences to a minimizer of the objective function is also obtained. In section 3 we deal with the case $1< p\le q+1$. We obtain fast convergence results concerning the potential energy, discrete velocity and subgradient. Moreover,  if $1<p<q+1$ strong convergence results for the sequence generated by \eqref{algo} to the minimum norm minimizer of the objective function is shown. Further, in case the stepsize parameter $\l_k\equiv 1$ we obtain full strong convergence of the sequences generated by Algorithm \eqref{algo} and improved convergence rates for the function values and velocity. Finally we conclude our paper by underlying some possible further researches.

\section{Convergence rates and weak convergence for the case $ q+1\le p$}
In this section we analyze the weak convergence properties of the sequence generated by the algorithm \eqref{algo}. We obtain fast convergence to zero of the discrete velocity and subgradient. We also show that the function values in the generated sequences converge to the global minimum of the objective function $f.$
Even more, the variable stepsize parameter $\l_k=\l k^\d,\,\l,\d>0$ allows to obtain the estimate of order $\mathcal{O}(k^{-q-\d-1})$ for the decay $f(x_k)-\min_{\cH} f$ which can  be arbitrary large, depending on parameter $\d.$

\subsection{Convergence rates}
Concerning fast convergence of the function values, discrete velocity and subgradient, we have the following result.
\begin{Theorem}\label{convergencealgorithm}
Assume that $0<q\le 1$, $q+1\le p$, $\l_k=\l k^\d,\,\l>0,\,\d\ge 0$ and let $(x_k)$ be a sequence generated by \eqref{algo}. For every $k\ge 2$ let us denote $u_k$ the element  from $\p f(x_k)$ that satisfies \eqref{formx} with equality, i.e.,
$$x_{k}=\a_{k-1}(x_{k-1}-x_{k-2})-\l_{k-1} u_k +\left(1- c_{k-1}\right) x_{k-1}.$$  Then   the following  results are valid.
\begin{enumerate}
\item[(i)] If $\a>0,\,\d\ge0,$ $0<q<1,\,q+1<p\le 2$ and for $p=2$ one has $c>q(1-q)$, or  $\a>3,\,0\le\d<\a-3,$ $q=1$ and  $p>2$ then
\[f(x_k)-\min_{\cH} f=\mathcal{O}(k^{-q-\d-1}),\,\|x_{k}-x_{k-1}\|=\mathcal{O}(k^{-\frac{q+1}{2}})\mbox{ and }\|u_k\|=o(k^{-\frac{q+1}{2}-\d})\mbox{ as } k\to+\infty.\]
Further,
 \[\ds\sum_{k=1}^{+\infty} k^{q+\d}(f(x_k)-\min_{\cH} f)<+\infty,\,\ds\sum_{k=1}^{+\infty} k\|x_k-x_{k-1}\|^2<+\infty\mbox{ and }\ds\sum_{k=2}^{+\infty} k^{q+2\d+1}\|u_k\|^2<+\infty.\]
\item[(ii)] If $\a>0,\,\d\ge0,$ $0<q<1$ and $q+1\le p$, or  $\a>3,\,0\le\d<\a-3,$ $q=1$ and  $p\ge 2$ then for all $s\in\left]\frac12,\frac{q+1}{2}\right[$ one has
$$f(x_k)-\min_{\cH} f=o(k^{-2s-\d}),\,\|x_{k}-x_{k-1}\|=o(k^{-s})\mbox{ and }\|u_k\|=o(k^{-s-\d})\mbox{ as } k\to+\infty.$$
Further,
 $$\ds\sum_{k=1}^{+\infty} k^{2s+\d-1}(f(x_k)-\min_{\cH} f)<+\infty,\,\ds\sum_{k=1}^{+\infty} k^{2s-q}\|x_k-x_{k-1}\|^2<+\infty\mbox{ and }\ds\sum_{k=2}^{+\infty} k^{2s+2\d}\|u_k\|^2<+\infty.$$
\end{enumerate}
\end{Theorem}

\begin{proof} Given $x^*\in\argmin f$, set $f^*=f(x^*)=\min_{\cH} f$.

For $k\ge 2$, consider the discrete energy
\begin{align}\label{discreteenergy}
E_k=&\mu_{k-1}(f(x_{k-1})-f^*)+\|a_{k-1}(x_{k-1}-x^*)+b_{k-1}(x_k-x_{k-1}+\l_{k-1}u_k)\|^2\\
\nonumber&+\nu_{k-1}\|x_{k-1}-x^*\|^2+\s_{k-1}\|x_{k-1}\|^2,
\end{align}
where  $a_k=ak^{r-1},\,b_k=k^r$, $r\in \left(\frac12,\frac{q+1}{2}\right],\,2r+\d<a$,
 $\mu_k:=(2b_{k}^2-2a_{k}b_{k})\l_{k},$ $\nu_k:=-\a_{k+1}a_{k+1}b_{k+1} -a_{k}^2+a_{k}b_{k}$ and $\sigma_k=\a_{k+1}b_{k+1}^2c_{k+1}$, for all $k\ge 1.$

 If $q=1$, hence $\a>3,$ we also assume that $a<\a-1,$ hence $\d<\a-1-2r.$

Let us develop $E_k$. We show first, that there exists $k_0\ge 1$ such that the coefficients $\mu_k,\nu_k$ and $\s_k$ are nonnegative for all $k\ge k_0.$

According to the form of $(a_k)$ and $(b_k)$,  there exists  $k_1\ge 1$ such that  $b_k\ge a_k$ for all $k \geq k_1$, hence
\begin{equation}\label{signmu}
\mu_k=(2b_{k}^2-2a_{k}b_{k})\l_{k}\ge 0\mbox{ for all }k\ge k_1\mbox{ and }\mu_k=\mathcal{O}(k^{2r+\d})\mbox{ as }k\to+\infty.
\end{equation}
\vskip0.3cm

Obviously $\nu_k=-\a_{k+1}a_{k+1}b_{k+1} -a_{k}^2+a_{k}b_{k}=-a(k+1)^{2r-1}+\a a(k+1)^{2r-1-q}-a^2k^{2r-2}+a k^{2r-1}$ and  we show that $\phi(x,r)=-a(x+1)^{2r-1}+\a a(x+1)^{2r-1-q}-a^2x^{2r-2}+a x^{2r-1}\ge 0$ for $x$ big enough and that $\phi(x,r)=\mathcal{O}(x^{2r-1-q})$ as $x\to+\infty.$
Indeed, one has
\begin{align*}
\lim_{x\to+\infty}\frac{\phi(x,r)}{x^{2r-1-q}}&=\lim_{x\to+\infty}\frac{a x^{2r-1}-a(x+1)^{2r-1}+\a a(x+1)^{2r-1-q}-a^2x^{2r-2}}{x^{2r-1-q}}\\
&=\lim_{x\to+\infty}\left(\frac{a -a\left(1+\frac{1}{x}\right)^{2r-1}}{x^{-q}}+\a a\left(1+\frac{1}{x}\right)^{2r-1-q}-a^2x^{q-1}\right)\\
&=\lim_{x\to+\infty}\left(-\frac{a(2r-1)}{q}\frac{\left(1+\frac{1}{x}\right)^{2r-2}}{x^{1-q}}+\a a-a^2x^{q-1}\right)=L,
\end{align*}
where $L=\a a>0$ if $q<1$ and $L=-a(2r-1)+\a a-a^2$ if $q=1.$ However, if $q=1$ one has $a<\a-1$ and consequently  $0< \a+1-a-2r$ hence $L>0$ also in this case.

Hence, there exists $k_2\ge 1$ such that for all $\frac12< r\le \frac{q+1}{2}$ one has
\begin{equation}\label{signnu}
\nu_k=-\a_{k+1}a_{k+1}b_{k+1} -a_{k}^2+a_{k}b_{k}\ge 0,\mbox{ for all }k\ge k_2\mbox{ and }\nu_k=\mathcal{O}(k^{2r-1-q})\mbox{ as }k\to+\infty.
\end{equation}

Finally, it is obvious that there exists $k_3\ge 1$ such that $\a_{k+1}b_{k+1}^2c_{k+1}=c\left(1-\frac{\a}{(k+1)^q}\right)(k+1)^{2r-p}\ge 0$ for all $k\ge k_3$, hence for all $\frac12< r\le \frac{q+1}{2}$ one has
\begin{equation}\label{signsigma}
\sigma_k=\a_{k+1}b_{k+1}^2c_{k+1}\ge 0\mbox{ for all }k\ge k_3\mbox{ and }\sigma_k=\mathcal{O}(k^{2r-p})\mbox{ as }k\to+\infty.
\end{equation}
Now, take $k_0=\max(k_1,k_2,k_3)$ and one has $\m_k,\nu_k,\s_k\ge0$ for all $k\ge k_0.$
\vskip0.3cm
For simplicity let us denote
$v_k=\|a_{k-1}(x_{k-1}-x^*)+b_{k-1}(x_k-x_{k-1}+\l_{k-1}u_k)\|^2.$
Then,
\begin{align}\label{forEk}
v_k=&a_{k-1}^2\|x_{k-1}-x^*\|^2+b_{k-1}^2\|x_k-x_{k-1}\|^2+b_{k-1}^2\l_{k-1}^2\|u_k\|^2+2a_{k-1}b_{k-1}\<x_k-x_{k-1},x_{k-1}-x^*\>\\ \nonumber
&+2a_{k-1}b_{k-1}\l_{k-1}\<u_k,x_{k-1}-x^*\>+2b_{k-1}^2\l_{k-1}\<u_k,x_k-x_{k-1}\>.
\end{align}
Further
$$
2a_{k-1}b_{k-1}\<x_k-x_{k-1},x_{k-1}-x^*\>=a_{k-1}b_{k-1}(\|x_k-x^*\|^2-\|x_k-x_{k-1}\|^2-\|x_{k-1}-x^*\|^2)$$
and
$$2a_{k-1}b_{k-1}\l_{k-1}\<u_k,x_{k-1}-x^*\>=2a_{k-1}b_{k-1}\l_{k-1}\<u_k,x_{k}-x^*\>-2a_{k-1}b_{k-1}\l_{k-1}\<u_k,x_{k}-x_{k-1}\>.$$
Consequently, \eqref{forEk} becomes
\begin{align}\label{forEk1}
v_k=&a_{k-1}b_{k-1}\|x_k-x^*\|^2+(a_{k-1}^2-a_{k-1}b_{k-1})\|x_{k-1}-x^*\|^2+(b_{k-1}^2-a_{k-1}b_{k-1})\|x_k-x_{k-1}\|^2\\
\nonumber&+b_{k-1}^2\l_{k-1}^2\|u_k\|^2+2a_{k-1}b_{k-1}\l_{k-1}\<u_k,x_{k}-x^*\>+(2b_{k-1}^2-2a_{k-1}b_{k-1})\l_{k-1}\<u_k,x_k-x_{k-1}\> .
\end{align}
Let us proceed similarly with $v_{k+1}$. First notice that
from (\ref{formx}) we have
$$v_{k+1}=\|a_{k}(x_{k}-x^*)+b_{k}(\a_k(x_k-x_{k-1})-c_k x_k)\|^2.$$
Therefore, after development we get
\begin{align}\label{forEk+1}
v_{k+1}=&a_k^2\|x_k-x^*\|^2+\a_k^2b_k^2\|x_k-x_{k-1}\|^2+b_k^2c_k^2\|x_k\|^2+2\a_k a_kb_k\<x_k-x_{k-1},x_k-x^*\> \\
\nonumber&-2\a_kb_k^2c_k\<x_k-x_{k-1},x_k\>-2a_kb_kc_k\<x_k,x_k-x^*\>.
\end{align}
Further,
\begin{eqnarray*}
&&2\a_k a_kb_k\<x_k-x_{k-1},x_k-x^*\>=-\a_k a_k b_k(\|x_{k-1}-x^*\|-\|x_k-x_{k-1}\|^2-\|x_k-x^*\|^2) \\
&&-2\a_k b_k^2c_k\<x_k-x_{k-1},x_k\>=\a_k b_k^2c_k(\|x_{k-1}\|^2-\|x_k-x_{k-1}\|^2-\|x_k\|^2)
\\
&&-2a_kb_kc_k\<x_k,x_k-x^*\>=a_k b_kc_k(\|x^*\|^2-\|x_k-x^*\|^2-\|x_k\|^2).
\end{eqnarray*}
Hence, (\ref{forEk+1}) yields
\begin{align}\label{forEk+11}
v_{k+1}&=(a_k^2+\a_k a_k b_k-a_k b_kc_k)\|x_k-x^*\|^2-\a_k a_kb_k\|x_{k-1}-x^*\|^2
\\
\nonumber& +(\a_k^2b_k^2+\a_k a_k b_k-\a_k b_k^2c_k)\|x_k-x_{k-1}\|^2+(b_k^2c_k^2-\a_kb_k^2c_k -a_kb_kc_k)\|x_k\|^2\\
\nonumber&+\a_k b_k^2c_k\|x_{k-1}\|^2+a_k b_kc_k\|x^*\|^2.
\end{align}
Hence, \eqref{forEk+11} and \eqref{forEk1} lead to
\begin{align}\label{energdif}
v_{k+1}-v_k=&(a_k^2+\a_k a_kb_k-a_kb_kc_k-a_{k-1}b_{k-1})\|x_k-x^*\|^2\\
\nonumber&+(-\a_k a_kb_k -a_{k-1}^2+a_{k-1}b_{k-1})\|x_{k-1}-x^*\|^2\\
\nonumber& +(\a_k^2b_k^2+\a_k a_kb_k-\a_kb_k^2c_k-b_{k-1}^2+a_{k-1}b_{k-1})\|x_k-x_{k-1}\|^2\\
\nonumber&+(b_k^2c_k^2-\a_kb_k^2c_k -a_kb_kc_k)\|x_k\|^2+\a_kb_k^2c_k\|x_{k-1}\|^2-b_{k-1}^2\l_{k-1}^2\|u_k\|^2\\
\nonumber&+2a_{k-1}b_{k-1}\l_{k-1}\<u_k,x^*-x_{k}\>+(2b_{k-1}^2-2a_{k-1}b_{k-1})\l_{k-1}\<u_k,x_{k-1}-x_k\>\\
\nonumber&+a_kb_kc_k\|x^*\|^2.
\end{align}
From the subgradient inequality we have
$$\<u_k,x^*-x_{k}\>\le f^*-f(x_k)\mbox{ and }\<u_k,x_{k-1}-x_k\>\le f(x_{k-1})-f(x_k).$$

Consequently, we get for all $k>k_0$ that
\begin{align}\label{forfrate}
&2a_{k-1}b_{k-1}\l_{k-1}\<u_k,x^*-x_{k}\>+(2b_{k-1}^2-2a_{k-1}b_{k-1})\l_{k-1}\<u_k,x_{k-1}-x_k\>\\
\nonumber&\le(2b_{k-1}^2-2a_{k-1}b_{k-1})\l_{k-1}(f(x_{k-1})-f^*)-2b_{k-1}^2\l_{k-1}(f(x_k)-f^*)\\
\nonumber&=\mu_{k-1}(f(x_{k-1})-f^*)-(\mu_k+(2b_{k-1}^2\l_{k-1}-2b_{k}^2\l_{k}+2a_{k}b_{k}\l_{k}))(f(x_{k})-f^*).
\end{align}

Let us denote  $m_k:=2b_{k-1}^2\l_{k-1}-2b_{k}^2\l_{k}+2a_{k}b_{k}\l_{k}$  and let us show that for all $\frac12< r\le \frac{q+1}{2}$ one has $2b_{k-1}^2\l_{k-1}-2b_{k}^2\l_{k}+2a_{k}b_{k}\l_{k}\ge 0$ for all $k\ge 1$. We can write equivalently as
$ k^{2r+\d}-ak^{2r+\d-1}-(k-1)^{2r+\d} \leq 0 $ for all $k\ge 1$.
Since $  2 r +\d< a $, by convexity of the function
$x\mapsto x^{2r+\d}$, the gradient differential inequality gives
$$
(x-1)^{2r+\d} \geq x^{2r+\d} -(2r+\d) x^{2r+\d -1} \geq x^{2r+\d} -a x^{2r+\d -1}
$$
and the claim follows.
Hence,
\begin{equation}\label{signm}
m_k\ge 0\mbox{ for all }k\ge k_0\mbox{ and observe that }m_k=\mathcal{O}(k^{2r+\d-1})\mbox{ as }k\to+\infty.
\end{equation}

Combining \eqref{energdif} and \eqref{forfrate} we get for all $k\ge k_0$ that
\begin{align}\label{forfrate1}
&v_{k+1}-v_k+\mu_k(f(x_{k})-f^*)-\mu_{k-1}(f(x_{k-1})-f^*)+m_k(f(x_{k})-f^*)\le\\
\nonumber&(a_k^2+\a_k a_kb_k-a_kb_kc_k-a_{k-1}b_{k-1})\|x_k-x^*\|^2\\
\nonumber&+(-\a_k a_kb_k -a_{k-1}^2+a_{k-1}b_{k-1})\|x_{k-1}-x^*\|^2\\
\nonumber& +(\a_k^2b_k^2+\a_k a_kb_k-\a_kb_k^2c_k-b_{k-1}^2+a_{k-1}b_{k-1})\|x_k-x_{k-1}\|^2\\
\nonumber&+(b_k^2c_k^2-\a_kb_k^2c_k -a_kb_kc_k)\|x_k\|^2+\a_kb_k^2c_k\|x_{k-1}\|^2-b_{k-1}^2\l_{k-1}^2\|u_k\|^2\\
\nonumber&+a_kb_kc_k\|x^*\|^2.
\end{align}
Let us analyze now the sign of the coefficients of the right hand side of \eqref{forfrate1}.
We have,
\begin{align*}
a_k^2+\a_k a_kb_k-a_kb_kc_k-a_{k-1}b_{k-1}&=(\a_{k+1}a_{k+1}b_{k+1} +a_{k}^2-a_{k}b_{k})\\
\nonumber&+(\a_k a_kb_k-a_kb_kc_k-a_{k-1}b_{k-1}-\a_{k+1}a_{k+1}b_{k+1}+a_{k}b_{k})\\
\nonumber& =-\nu_k-n_k,
\end{align*}
where $n_k:=-(\a_k a_kb_k-a_kb_kc_k-a_{k-1}b_{k-1}-\a_{k+1}a_{k+1}b_{k+1}+a_{k}b_{k}).$

Now, one has $$n_k=-(2ak^{2r-1}-\a ak^{2r-1-q}-ack^{2r-1-p}-a(k-1)^{2r-1}-
a(k+1)^{2r-1}+\a a(k+1)^{2r-1-q}).$$
We show that  for all $\frac12< r\le \frac{q+1}{2}$ one has
$$\phi(x,r)=-2ax^{2r-1}+\a ax^{2r-1-q}+acx^{2r-1-p}+a(x-1)^{2r-1}+a(x+1)^{2r-1}-\a a(x+1)^{2r-1-q}\ge 0$$
for $x$ big enough.

Indeed, if $q=1$ then one can take $r=1$ and we have $\phi(x,1)=acx^{1-p}>0$. Otherwise, for $\frac12<r<1$ one has
\begin{align}\label{flimhalfsign}
&\lim_{x\to+\infty}\frac{a(x-1)^{2r-1}+a(x+1)^{2r-1}-2ax^{2r-1}}{x^{2r-3}}=\lim_{x\to+\infty}\frac{a\left(1-\frac{1}{x}\right)^{2r-1}+a\left(1+\frac{1}{x}\right)^{2r-1}-2a}{x^{-2}}\\
\nonumber&=\lim_{x\to+\infty}\frac{a(2r-1)}{x^2}\frac{\left(1-\frac{1}{x}\right)^{2r-2}-\left(1+\frac{1}{x}\right)^{2r-2}}{-2x^{-3}}=\lim_{x\to+\infty}\frac{a(2r-1)}{-2}\frac{\left(1-\frac{1}{x}\right)^{2r-2}-\left(1+\frac{1}{x}\right)^{2r-2}}{x^{-1}}\\
\nonumber&=\lim_{x\to+\infty}\frac{a(2r-1)(2r-2)}{-2x^2}\frac{\left(1-\frac{1}{x}\right)^{2r-3}+\left(1+\frac{1}{x}\right)^{2r-3}}{-x^{-2}}=a(2r-1)(2r-2)<0.
\end{align}
Consequently, there exists $C_1>0$ such that
\begin{equation}\label{fhalfsign}
a(x-1)^{2r-1}+a(x+1)^{2r-1}-2ax^{2r-1}\ge -C_1 x^{2r-3}\mbox{ for }x\mbox{ big enough.}
\end{equation}

Further, if $r=\frac{q+1}{2}$, then $\a ax^{2r-1-q}-\a a(x+1)^{2r-1-q}=0$, otherwise
\begin{align}\label{slimhalfsign}
&\lim_{x\to+\infty}\frac{ \a ax^{2r-1-q}-\a a(x+1)^{2r-1-q}}{x^{2r-2-q}}=\lim_{x\to+\infty}\a a\frac{1-\left(1+\frac{1}{x}\right)^{2r-1-q}}{x^{-1}}\\
\nonumber&=-\a a\lim_{x\to+\infty}(2r-1-q)\left(1+\frac{1}{x}\right)^{2r-2-q}=\a a(1+q-2r)>0.
\end{align}
Consequently, there exists $C_2>0$ such that
\begin{equation}\label{shalfsign}
\a ax^{2r-1-q}-\a a(x+1)^{2r-1-q}\ge C_2x^{2r-2-q}\mbox{ for }x\mbox{ big enough.}
\end{equation}

From the above relations one can deduce the following:
\begin{enumerate}
\item[(N1)] If $q=1$ and $r=1$ we have $p>2$ and $\phi(x,r)=acx^{1-p}>0$, hence $\phi(x,r)=\mathcal{O}(x^{1-p})$ as $x\to +\infty.$
\item[(N2)] If $q=1$ and $\frac12<r<1$ then $p\ge 2$ and according to  \eqref{flimhalfsign} and \eqref{slimhalfsign} and the fact that $\a>3$ we have
\begin{align*}
&\lim_{x\to+\infty}\frac{(a(x-1)^{2r-1}+a(x+1)^{2r-1}-2ax^{2r-1})+(\a ax^{2r-1-q}-\a a(x+1)^{2r-1-q})}{x^{2r-3}}\\
&=a(2r-1)(2r-2)+\a a(2-2r)=a(2-2r)(\a+1-2r)>0.
\end{align*}
Hence, $\phi(x,r)\ge Cx^{2r-3}+acx^{2r-1-p}$ for some $C>0$ and for $x$ big enough.
Consequently, also in this case $\phi(x,r)>0$ if $x$ is big enough and since $p\ge 2$ one has $\phi(x,r)=\mathcal{O}(x^{2r-3})$ as $x\to +\infty.$
\item[(N3)] If $0<q<1,\,r=\frac{q+1}{2}$ then $q+1<p\le 2$ and according to \eqref{fhalfsign} one has
\begin{align*}
&(a(x-1)^{2r-1}+a(x+1)^{2r-1}-2ax^{2r-1})+(\a ax^{2r-1-q}-\a a(x+1)^{2r-1-q})\\
&\ge-C_1 x^{q-2}\mbox{ for }x\mbox{ big enough.}
\end{align*}
Hence, $\phi(x,r)\ge acx^{q-p}-C_1 x^{q-2}\mbox{ for }x\mbox{ big enough.}$ Obviously $\phi(x,r)>0$ if $p<2$ and $x$ is big enough. Further,
if $p=2$ then \eqref{flimhalfsign} gives $\lim_{x\to\infty}\frac{\phi(x,r)}{x^{q-2}}=a(c+q(q-1))>0$, hence one has $\phi(x,r)>0$ if  $x$ is big enough.

Observe that in this case one has $\phi(x,r)=\mathcal{O}(x^{q-p})$ as $x\to +\infty.$
\item[(N4)] If $0<q<1,\,\frac12<r<\frac{q+1}{2}$ then $q+1\le p$ and according to \eqref{fhalfsign} and \eqref{shalfsign} one has
\begin{align*}
&\phi(x,r)\ge-C_1 x^{2r-3}+C_2x^{2r-2-q}+acx^{2r-1-p}\ge Cx^{2r-2-q}\mbox{ for some }C>0\mbox{ and for }x\mbox{ big enough.}
\end{align*}
Consequently, also in this case $\phi(x,r)>0$ if $x$ is big enough and observe that $\phi(x,r)=\mathcal{O}(x^{2r-2-q})$ as $x\to +\infty.$
\end{enumerate}
We conclude that there exist  $K_1\ge k_0$ such that for all $\frac12< r\le \frac{q+1}{2}$ one has
\begin{equation}\label{forstrxk}
n_k\ge 0,\mbox{ for all }k\ge K_1\mbox{ and the appropriate estimates emphasized at (N1)-(N4) hold.}
\end{equation}

For the coefficient of discrete velocity $\|x_k-x_{k-1}\|^2$ we have
\begin{align*}
\a_k^2b_k^2+\a_k a_kb_k-\a_kb_k^2c_k-b_{k-1}^2+a_{k-1}b_{k-1}&= k^{2r}-(k-1)^{2r}-2\a k^{2r-q} +ak^{2r-1}\\
&+a(k-1)^{2r-1}+\a^2 k^{2r-2q} -\a a k^{2r-q-1}-ck^{2r-p}\\
&+\a  ck^{2r-q-p}.
\end{align*}
We show that  for all $\frac12< r\le \frac{q+1}{2}$ one has
\begin{align*}
\phi(x,r)=&(x-1)^{2r}- x^{2r}+2\a x^{2r-q} -ax^{2r-1}-a(x-1)^{2r-1}-\a^2 x^{2r-2q} +\a a x^{2r-q-1}+cx^{2r-p}\\
&-\a  cx^{2r-q-p}\ge 0,\mbox{ if }x\mbox{ is big enough.}
\end{align*}
Even more, $\phi(x,r)=\mathcal{O}(x^{2r-q})$ as $x\to+\infty.$

Indeed
\begin{align*}
&\lim_{x\to+\infty}\frac{(x-1)^{2r}- x^{2r}+2\a x^{2r-q} -ax^{2r-1}-a(x-1)^{2r-1}-\a^2 x^{2r-2q} +\a a x^{2r-q-1}-\a  cx^{2r-q-p}}{x^{2r-q}}\\
&=\lim_{x\to+\infty}\frac{\left(1-\frac{1}{x}\right)^{2r}- 1 -ax^{-1}-\frac{a}{x}\left(1-\frac{1}{x}\right)^{2r-1}}{x^{-q}}+2\a \\
&=\lim_{x\to+\infty}\frac{-\frac{2r}{q}\left(1-\frac{1}{x}\right)^{2r-1}-\frac{a}{q}-\frac{a}{q}\left(1-\frac{1}{x}\right)^{2r-1}}{x^{1-q}}+2\a=L.
\end{align*}
Obviously, $L=2\a>0$ if $q<1$ and $L=-2r-2a+2\a$ if $q=1.$ But then $\a>3,$ $a< \a-1$ and $r\le 1$, hence also in this case $L=-2r-2a+2\a>0.$
Consequently, there exists $C>0$ such that
$$\phi(x,r)\ge Cx^{2r-q}+cx^{2r-p}>0\mbox{ if }x\mbox{ is big enough}$$
and since $p>1$ one has
$$\phi(x,r)=\mathcal{O}(x^{2r-q})\mbox{ as }x\to+\infty.$$

We conclude that there exist $K_2\ge k_0$ such that for all $\frac12< r\le\frac{q+1}{2}$ one has
\begin{equation}\label{forspeed}
\eta_k\ge0,\mbox{ for all }k\ge K_2\mbox{ and }\eta_k=\mathcal{O}(k^{2r-q})\mbox{ as }k\to+\infty,
\end{equation}
where $\eta_k:=-\a_k^2b_k^2-\a_k a_kb_k+\a_kb_k^2c_k+b_{k-1}^2-a_{k-1}b_{k-1}.$

The coefficient of $\|x_{k-1}\|^2$ is $\s_{k-1}=\a_kb_k^2c_k$, hence we write the coefficient of $\|x_k\|^2$ as
$$b_k^2c_k^2-\a_kb_k^2c_k -a_kb_kc_k=-\sigma_k+(b_k^2c_k^2+\a_{k+1}b_{k+1}^2c_{k+1}-\a_kb_k^2c_k -a_kb_kc_k).$$

We have
\begin{align*}
b_k^2c_k^2+\a_{k+1}b_{k+1}^2c_{k+1}-\a_kb_k^2c_k -a_kb_kc_k&=c^2k^{2r-2p}+ c(k+1)^{2r-p}-\a c(k+1)^{2r-p-q}\\
&-ck^{2r-p}+\a ck^{2r-p-q}-ack^{2r-1-p}.
\end{align*}
We show that  for all $\frac12< r\le \frac{q+1}{2}$ one has
$$\phi(x,r)=c(x+1)^{2r-p}-cx^{2r-p}-\a c(x+1)^{2r-p-q}+\a cx^{2r-p-q}-acx^{2r-1-p}+c^2x^{2r-2p}\le 0$$
for $x$ big enough. Even more, $\phi(x,r)=\mathcal{O}(x^{2r-p-1})$ as $x\to+\infty.$

Indeed, since $1<2r\le q+1\le p$ we have,
\begin{align*}
&\lim_{x\to+\infty}\frac{c(x+1)^{2r-p}-cx^{2r-p}-\a c(x+1)^{2r-p-q}+\a cx^{2r-p-q}-acx^{2r-1-p}}{x^{2r-p-1}}\\
&=\lim_{x\to+\infty}\left(\frac{c\left(1+\frac{1}{x}\right)^{2r-p}- c}{x^{-1}} +\a c\frac{-\left(1+\frac{1}{x}\right)^{2r-p-q}+1}{x^{q-1}}\right)-ac \\
&=\lim_{x\to+\infty}\frac{c(2r-p)x^{-2}\left(1+\frac{1}{x}\right)^{2r-p-1}}{x^{-2}}-ac=c(2r-p-a)<0.
\end{align*}
Obviously, there exists $C>0$ such that
$$c(x+1)^{2r-p}-cx^{2r-p}-\a c(x+1)^{2r-p-q}+\a cx^{2r-p-q}-acx^{2r-1-p}\le -Cx^{2r-p-1}$$
for $x$ big enough, and from the fact that $p>1$ we get that $\phi(x,r)\le 0$ for $x$ big enough.

We conclude that there exist $K_3\ge k_0$ such that for all $\frac12< r\le\frac{q+1}{2}$ one has
\begin{equation}\label{signxk}
s_k\ge 0\mbox{ for all }k\ge K_3\mbox{ and  }s_k=\mathcal{O}(k^{2r-p-1})\mbox{ as }k\to+\infty,
\end{equation}
where $s_k:=-(b_k^2c_k^2+\a_{k+1}b_{k+1}^2c_{k+1}-\a_kb_k^2c_k -a_kb_kc_k).$ 
Let $K_0=\max(K_1,K_2,K_3).$

Combining \eqref{forfrate1}, \eqref{forstrxk}, \eqref{forspeed} and \eqref{signxk} we obtain that
for all $k\ge K_0$ and $r\in\left(\frac12,\frac{q+1}{2}\right]$ it holds
\begin{align}\label{forfrate2}
v_{k+1}-v_k\le& -(\mu_k(f(x_{k})-f^*)-\mu_{k-1}(f(x_{k-1})-f^*))-m_k(f(x_{k})-f^*)\\
\nonumber&-(\nu_k\|x_k-x^*\|^2-\nu_{k-1}\|x_{k-1}-x^*\|^2)-n_k\|x_k-x^*\|^2\\
\nonumber&-(\sigma_k\|x_k\|^2-\sigma_{k-1}\|x_{k-1}\|^2)-s_k\|x_k\|^2\\
\nonumber& -\eta_k\|x_k-x_{k-1}\|^2-b_{k-1}^2\l_{k-1}^2\|u_k\|^2+ a_kb_kc_k\|x^*\|^2.
\end{align}
Consequently
\begin{align}\label{energi}
E_{k+1}&-E_k+m_k(f(x_{k})-f^*)+\eta_k\|x_k-x_{k-1}\|^2+b_{k-1}^2\l_{k-1}^2\|u_k\|^2+n_k\|x_k-x^*\|^2+s_k\|x_k\|^2\\
\nonumber&\le a_kb_kc_k\|x^*\|^2=ac\|x^*\|^2 k^{2r-1-p},
\end{align}
for all $k\ge K_0.$

Now in concordance  to the hypotheses of the theorem we take $r<\frac{q+1}{2}$ if $p=q+1$, consequently one has $2r-1-p<-1$, hence $$ac\|x^*\|^2\sum_{k\ge K_0}k^{2r-1-p}<+\infty.$$
By summing up \eqref{energi} from $k=K_0$ to $k=n>K_0$, we obtain that there exists $C_1>0$ such that
$$E_{n+1}\le C_1,$$
consequently
$$\mu_n(f(x_{n})-f^*)\le C_1,\mbox{ hence }f(x_{n})-f^*=\mathcal{O}(n^{-2r-\d})\mbox{ as }n\to+\infty,$$
$$\nu_n\|x_n-x^*\|^2\le C_1,\mbox{ hence }\|x_n-x^*\|^2=\mathcal{O}(n^{q+1-2r})\mbox{ as }n\to+\infty,$$
$$\s_n\|x_n\|^2\le C_1,\mbox{ hence }\|x_n\|^2=\mathcal{O}(n^{p-2r})\mbox{ as }n\to+\infty$$
and
$$\sup_{n\ge 1}\|an^{r-1}(x_n-x^*)+n^{r}(x_{n+1}-x_{n}+\l n^{\d}u_{n+1})\|<+\infty.$$
Further,
\begin{align*}&\sum_{k= K_0}^n m_k(f(x_{k})-f^*)\le C_1,\mbox{ hence according to \eqref{signm} one has }\sum_{k\ge 1}k^{2r+\d-1}(f(x_{k})-f^*)<+\infty,\\
&\sum_{k= K_0}^n \eta_k\|x_k-x_{k-1}\|^2\le C_1,\mbox{ hence according to \eqref{forspeed} one has }\sum_{k\ge 1}k^{2r-q}\|x_k-x_{k-1}\|^2<+\infty,\\
&\sum_{k= K_0}^n b_{k-1}^2\l_{k-1}^2\|u_k\|^2\le C_1,\mbox{ hence one has }\sum_{k\ge 1}k^{2r+2\d}\|u_k\|^2<+\infty,\\
&\sum_{k= K_0}^n s_k\|x_k\|^2\le C_1,\mbox{ hence according to \eqref{signxk} one has }\sum_{k\ge 1}k^{2r-p-1}\|x_k\|^2<+\infty.
\end{align*}

Moreover, $\sum_{k= K_0}^n n_k\|x_k-x^*\|^2\le C_1,$ hence according to \eqref{forstrxk} one has
$$\sum_{k\ge 1}k^{q-p}\|x_k-x^*\|^2<+\infty,\mbox{ if  }r=\frac{q+1}{2}$$
and
$$\sum_{k\ge 1}k^{2r-2-q}\|x_k-x^*\|^2<+\infty,\mbox{ if  }r<\frac{q+1}{2}.$$

Since $\sum_{k\ge 1}k^{2r+2\d}\|u_k\|^2<+\infty$ one has $\|u_n\|=o(n^{-r-\d})$ as $n\to +\infty$ which  yields
$$\sup_{n\ge 1}\|an^{r-1}(x_n-x^*)+n^{r}(x_{n+1}-x_{n})\|<+\infty.$$
Combining the latter relation with the facts that $\|x_n-x^*\|^2=\mathcal{O}(n^{q+1-2r})\mbox{ as }n\to+\infty$ and $n^{r-1}\le n^{\frac{2r-q-1}{2}}$ we obtain
$$\|x_{n+1}-x_n\|=\mathcal{O}(n^{-r})\mbox{ as }n\to+\infty.$$

Let us show now, that for $\frac12<r<\frac{q+1}{2}$ one has $f(x_n)-f^*= o(n^{-2r-\d})$ and $\|x_n-x_{n-1}\|=o(n^{-r}).$

From \eqref{energi} we get
$$\sum_{k\ge 1} [(E_{k+1}-E_{k}]_+<+\infty,\mbox{ where }[s]_+=\max(s,0).$$
Therefore, the following  limit exists
\begin{equation}\label{foroest}
\lim_{k\to+\infty}(\|ak^{r-1}(x_k-x^*)+k^{r}(x_{k+1}-x_{k}+\l k^{\d}u_{k+1})\|^2+\s_k\|x_k\|^2+\mu_k (f(x_{k})-f^*)+\nu_k\|x_k-x^*\|^2).
\end{equation}
Note that according to \eqref{signsigma}, \eqref{signmu} and \eqref{signnu} one has $\s_k=\mathcal{O}(k^{2r-p}),\,\mu_k=\mathcal{O}(k^{2r+\d})$ and $\nu_k=\mathcal{O}(k^{2r-1-q}),$ respectively.

Further, if $r<\frac{q+1}{2}$ we have
 $\sum_{k\ge 1} k^{2r-2-q}\|x_k-x^*\|^2<+\infty$ and the following estimates also hold:  $\sum_{k\ge 1} k^{2r-q}\|x_k-x_{k-1}\|^2<+\infty$, $\sum_{k\ge 1} k^{2r+2\d}\|u_k\|^2< +\infty$, $\sum_{k\ge 1}k^{2r+\d-1}(f(x_{k})-f^*)< +\infty$ and finally  $\sum_{k\ge 1} k^{2r-1-p}\|x_k\|^2<+\infty$. Hence,
\begin{align}\label{foroest1}
&\sum_{k\ge 1}\frac{1}{k}(\|ak^{r-1}(x_k-x^*)+k^{r}(x_{k+1}-x_{k}+\l k^{\d}u_{k+1})\|^2+\s_k\|x_k\|^2+\mu_k (f(x_{k})-f^*)+\nu_k\|x_k-x^*\|^2)\\
\nonumber&\le \sum_{k\ge 1}2a^2k^{2r-3}\|x_k-x^*\|^2+\sum_{k\ge 1}4 k^{2r-1}\|x_{k+1}-x_{k}\|^2+\sum_{k\ge 1}4 \l^2(k+1)^{2r+2\d-1}\|u_{k+1})\|^2\\
\nonumber&+C\left(\sum_{k\ge 1}k^{2r-p-1}\|x_k\|^2+\sum_{k\ge 1}k^{2r+\d-1} (f(x_{k})-f^*)+\sum_{k\ge 1}k^{2r-2-q}\|x_k-x^*\|^2\right)<+\infty,
 \end{align}
for some constant $C>0.$

Combining the facts that $\sum_{k\ge 1}\frac{1}{k}=+\infty$ and $\|u_n\|=o(n^{-r-\d})$ as $n\to +\infty$ with \eqref{foroest1} and \eqref{foroest} we get
$$\lim_{k\to+\infty}(\|ak^{r-1}(x_k-x^*)+k^{r}(x_{k+1}-x_{k}+\l k^\d u_{k+1})\|^2+\s_k\|x_k\|^2+\mu_k (f(x_{k})-f^*)+\nu_k\|x_k-x^*\|^2)=0$$
and the claim follows.
\end{proof}
\begin{remark}\label{comACR1} Note that our analysis also works in case $c=0$. In that case we do not have Tikhonov regularization, hence one does not have to impose any assumption on $p$ in the hypotheses of Theorem \ref{convergencealgorithm} and the conclusion of the theorem remains valid.
This also shows that the choice  $q=1$  in \cite{ACR1} is not optimal.
Indeed, in the case $0<q<1$, according to Theorem \ref{convergencealgorithm}, arbitrary large convergence rate for the potential energy $f(x_k)-\min_{\cH}f$ can be obtained, for a fixed inertial parameter $\a>0.$ Note that this result does not hold in case $q=1$ (see \cite{ACR1} or Theorem \ref{convergencealgorithm}), since in this case the inertial parameter $\a_k$ and the stepsize parameter $\l_k$ are correlated.  Let us discuss this aspect more detailed. In one hand, in \cite{ACR1} and also in  Theorem \ref{convergencealgorithm}, for the constellation $q=1,$ $\a>3$ and $\l_k=\mathcal{O}(k^\d)\,\mbox{ as }k\to+\infty,$ for $\d<\a-3$ is obtained the rate $f(x_k)-\min_{\cH}f=o(k^{-2-\d}),\,\mbox{ as }k\to+\infty.$ This means that for a fixed $\a$ one can obtain at most $f(x_k)-\min_{\cH}f=\mathcal{O}(k^{1-\a}),\,\mbox{ as }k\to+\infty.$ On the other hand, according to Theorem \ref{convergencealgorithm}, for  $0<q<1$ and  $q+1<p\le 2$  one has $f(x_k)-\min_{\cH} f=\mathcal{O}(k^{-q-\d-1}),\,\mbox{ as }k\to+\infty$, which indeed can be arbitrary large. Even more, as we emphasized before, our proof works also when $c=0$,  and of course then the assumption $q+1<p\le 2$  can be dropped.
\end{remark}

\subsection{On weak convergence and boundedness of the generated sequences}

In this section we provide sufficient conditions that assure that the sequence $(x_k)$ generated by the algorithm \eqref{algo} converges weakly to a minimizer of $f.$ In order to continue our analysis we need the following lemma, which is an extension of Lemma 8.3 from \cite{ACR1}.

\begin{lemma}\label{lforwconv} Assume that $(a_k)_{k\ge 1},\,(\omega_k)_{k\ge 1}$ are nonnegative real sequences that after an index $k_0$ satisfy
$$a_{k+1}\le \left(1-\frac{\a}{k^q}\right) a_k+\omega_k,\mbox{ for all }k\ge k_0,$$
where $q\in\left]0,1\right]$ and for $q=1$ one has $\a>1.$
Assume further, that $\sum_{k\ge k_0}k^q\omega_k<+\infty.$ Then,
$$\sum_{k\ge 1} a_k<+\infty.$$
\end{lemma}
\begin{proof} We have
$k^qa_{k+1}-k^q a_k+\a a_k\le k^q\omega_k,\mbox{ for all }k\ge k_0.$
If $q=1$ then  $\a>1$ hence we have for all $k\ge k_0$ that
$k a_{k+1}-k a_k+\a a_k= ka_{k+1}-(k-1) a_k+(\a-1) a_k,$
consequently
 $$ka_{k+1}-(k-1) a_k+(\a-1) a_k\le k\omega_k,\mbox{ for all }k\ge k_0.$$
By summing up the latter relation from $k=k_0$ to $k=n>k_0$ we get
$$na_{n+1}+(\a-1)\sum_{k=k_0}^n a_k\le (k_0-1)a_{k_0}+\sum_{k=k_0}^n k\omega_k.$$
Now, we omit the term $na_{n+1}$ and we take the limit $n\to+\infty$ in order to show that
$$\sum_{k=k_0}^{+\infty} a_k\le \frac{(k_0-1)a_{k_0}}{\a-1}+\frac{1}{\a-1}\sum_{k=k_0}^{+\infty} k\omega_k<+\infty.$$
If $q<1$ then, since $\lim_{k\to+\infty}\frac{k^q-(k-1)^q}{k^{q-1}}=q>0,$
we conclude that there exists $C>0$ and $k_1\ge k_0$ such that $k^q-(k-1)^q\le C k^{q-1}$ for all $k\ge k_1.$

Hence, there exists $k_2\ge k_1$ such that for all $k\ge k_2$ one has
$$-k^q\ge -(k-1)^q-C k^{q-1}\ge -(k-1)^q-\frac{\a}{2}.$$

Consequently
$k^qa_{k+1}-(k-1)^q a_k+\frac{\a}{2} a_k\le k^q\omega_k,\mbox{ for all }k\ge k_2.$
By summing up the latter relation from $k=k_2$ to $k=n>k_2$ we get
$$n^q a_{n+1}+\frac{\a}{2}\sum_{k=k_2}^n a_k\le (k_2-1)^q a_{k_2}+\sum_{k=k_2}^n k^q\omega_k$$
and the conclusion follows.
\end{proof}
Now we can prove the weak convergence of the sequences generated by algorithm \eqref{algo} to a minimizer of the objective function $f.$

\begin{Theorem}\label{weakconvergencealgorithm}
Assume that  $\a>0,$ $0<q<1,0\le \d,\,q+1<p\le 2$ and for $p=2$ one has $c>q(1-q)$, or  $q=1$, $p>2$, $\a>3,\,0\le\d<\a-3.$  Then the sequence $(x_n)$ generated by \eqref{algo} converges weakly to a minimizer of $f.$
\end{Theorem}
\begin{proof} We use the Opial lemma (see \cite{opial}). To this purpose first we show that for all $x^*\in\argmin f$ the limit $\lim_{k\to+\infty}\|x_k-x^*\|$ exists. Let $x^*\in\argmin f$ and for all $k\ge 1$ consider the sequence $h_k=\frac12\|x_k-x^*\|^2.$ Then, by using \eqref{formx} we have
\begin{align}\label{forwconv}
h_{k+1}-h_k&=\frac12\|x_{k+1}-x_k\|^2+\<x_{k+1}-x_k,x_k-x^*\>\\
\nonumber&=\frac12\|x_{k+1}-x_k\|^2+\<\a_k(x_{k}-x_{k-1})-\l_k u_{k+1}-c_k x_k,x_k-x^*\>.
\end{align}
Further, one has
$$\<\a_k(x_{k}-x_{k-1}),x_k-x^*\>=\frac{\a_k}{2}(\|x_{k}-x_{k-1}\|^2+\|x_k-x^*\|^2-\|x_{k-1}-x^*\|^2),$$
$$\<-\l_k u_{k+1},x_k-x^*\>\le\frac{\l}{2}(k^{1-q}\|x_{k+1}-x_k\|^2+k^{q+2\d-1}\|u_{k+1}\|^2)$$
and
$$\<-c_k x_k,x_k-x^*\>=\frac{c_k}{2}(\|x^*\|^2-\|x_k\|^2-\|x_k-x^*\|^2).$$
Consequently, \eqref{forwconv} leads to
\begin{align}\label{forwconv1}
h_{k+1}-h_k&\le \a_k(h_k-h_{k-1})+\frac{\a_k}{2}\|x_{k}-x_{k-1}\|^2+\frac12\|x_{k+1}-x_k\|^2+\frac{\l}{2}k^{1-q}\|x_{k+1}-x_k\|^2\\
\nonumber&+\frac{\l}{2}k^{q+2\d-1}\|u_{k+1}\|^2+\frac{c_k}{2}\|x^*\|^2.
\end{align}

We use Lemma \ref{lforwconv} with $a_k=[h_k-h_{k-1}]_+$ and $\omega_k=\frac{\a_k}{2}\|x_{k}-x_{k-1}\|^2+\frac12\|x_{k+1}-x_k\|^2+\frac{\l}{2}k^{1-q}\|x_{k+1}-x_k\|^2+\frac{\l}{2}k^{q+2\d-1}\|u_{k+1}\|^2+\frac{c_k}{2}\|x^*\|^2.$ Hence, we need to show that $\sum_{k\ge 1}k^q\omega_k<+\infty.$

According to Theorem \ref{convergencealgorithm} (i) and the fact that $p>q+1$ we have
$$\sum_{k\ge 1}k\|x_{k}-x_{k-1}\|^2<+\infty,\,\sum_{k=1}^{+\infty} k^{q+2\d+1}\|u_k\|^2<+\infty\mbox{ and }\sum_{k\ge 1}k^q c_k=\sum_{k\ge 1}\frac{c}{k^{p-q}}<+\infty.$$
Now, it is obvious that $\sum_{k=1}^{+\infty} k^{2q+2\d-1}\|u_{k+1}\|^2<+\infty.$
Consequently, $\sum_{k\ge 1}k^q\omega_k<+\infty$ and by Lemma \ref{lforwconv} we get that
$$\sum_{k\ge 1}[h_k-h_{k-1}]_+<+\infty,$$
which shows that $\lim_{k\to+\infty}\|x_k-x^*\|$ exists.

Next we show that every weak sequential cluster point of $(x_k)$ belongs to $\argmin f.$
Indeed, let $x^*$ a weak sequential cluster point of $(x_k).$ Then there exists an increasing sequence of natural numbers $(k_n)$ with $k_n\to+\infty,\mbox{ as }n\to+\infty$, such that
$x_{k_n}\rightharpoonup x^*\mbox{ as }n\to+\infty,$ where "$\rightharpoonup$" denotes the convergence with respect of weak topology of $\cH.$ Since $f$ is convex and lower semicontinuous it is also lower semicontinuous with respect to the weak topology of $\cH$. Further, according to Theorem \ref{convergencealgorithm} one has $\lim_{n\to+\infty}f(x_{k_n})=\min_{\cH}f,$ hence
$$f(x^*)\le\liminf_{n\to+\infty}f(x_{k_n})=\min_{\cH}f,$$
which shows that $x^*\in\argmin f.$

Consequently, Opial's lemma yields that the sequence $(x_n)$  converges weakly to a minimizer of our objective function $f.$
\end{proof}
\begin{remark} Also here our analysis  remains valid in case $c=0$, hence in that case one may obtain the weak convergence of the sequences generated by Algorithm \eqref{algo} without any restriction imposed on the parameter $p.$
\end{remark}

According to Theorem \ref{weakconvergencealgorithm} in case   $\a>0,$ $0<q<1,0\le \d,\,q+1<p\le 2$ the sequence $(x_n)$ generated by \eqref{algo} is bounded. We show next that this result also holds in case  $1<p<q+1.$
\begin{Theorem}\label{boundedness}
Assume that  $\a>0,$ $0<q<1,0\le \d,\,1<p<q+1$.  Then the sequence $(x_n)$ generated by \eqref{algo} is bounded.
\end{Theorem}
\begin{proof} We use the energy functional and notations from the proof of Theorem \ref{convergencealgorithm} but we assume that  $\frac{p}{2}<r<\frac{q+1}{2}.$
Note that all the estimates from the proof of Theorem \ref{convergencealgorithm} concerning the coefficients $\mu_k,\nu_k,\s_k,\m_k,\eta_k$ remains valid.

Let us compute the order of $n_k.$
We have
$n_k=-(\a_k a_kb_k-a_kb_kc_k-a_{k-1}b_{k-1}-\a_{k+1}a_{k+1}b_{k+1}+a_{k}b_{k}),$ hence
\begin{align*}
n_k&=a((k+1)^{2r-1}+(k-1)^{2r-1}-2k^{2r-1})-a\a((k+1)^{2r-1-q}-k^{2r-1-q})+ack^{2r-1-p}\\
&=\mathcal{O}(k^{2r-3})+\mathcal{O}(k^{2r-2-q})+ack^{2r-1-p}.
\end{align*}
Consequently, $n_k>0$ for $k$ big enough and $n_k=\mathcal{O}(k^{2r-1-p})$ as $k\to+\infty.$

Further, we have $s_k=-(b_k^2c_k^2+\a_{k+1}b_{k+1}^2c_{k+1}-\a_kb_k^2c_k -a_kb_kc_k),$ hence
\begin{align*}
s_k&=ack^{2r-1-p}+c(k^{2r-p}-(k+1)^{2r-p})+\a c((k+1)^{2r-q-p}-k^{2r-q-p})+c^2k^{2r-2p}\\
&=ack^{2r-1-p}-c(2r-p)\mathcal{O}(k^{2r-1-p})-\a c(2r-q-p)\mathcal{O}(k^{2r-1-q-p})+\mathcal{O}(k^{2r-2p}).
\end{align*}
Since $a>2r+\d>2r-p$ we conclude that  $s_k>0$ for $k$ big enough and $s_k=\mathcal{O}(k^{2r-1-p})$ as $k\to+\infty.$
Consequently \eqref{energi} holds with these coefficients after an index $K_0$ big enough. By neglecting the nonegative term $m_k(f(x_{k})-f^*)+\eta_k\|x_k-x_{k-1}\|^2+b_{k-1}^2\l_{k-1}^2\|u_k\|^2+n_k\|x_k-x^*\|^2+s_k\|x_k\|^2$ in \eqref{energi} we get
\begin{align}\label{energibound}
E_{k+1}&-E_k\le ac\|x^*\|^2 k^{2r-1-p},\mbox{ for all }k\ge K_0.
\end{align}
By summing up \eqref{energibound} from $k=K_0$ to $k=n>K_0$, we obtain that
$$E_{n+1}\le ac\|x^*\|^2 \sum_{k=K_0}^n k^{2r-1-p}+E_{K_0},$$
and since $\sum_{k=K_0}^n k^{2r-1-p}=\mathcal{O}(n^{2r-p})$ as $n\to+\infty$ we conclude that there exists $C_0>0$ such that
$E_{n+1}\le C_0 n^{2r-p}.$ In particular we have $\s_n\|x_n\|^2\le C_0 n^{2r-p}$ and according to \eqref{signsigma} $\s(n)=\mathcal{O}(n^{2r-p})$, hence $x_n$ is bounded.
\end{proof}

\section{Convergence rates  and strong convergence results for the  case $p\le q+1$}
We continue the present section by emphasizing the main idea behind the Tikhonov regularization, which will assure strong convergence results for the sequence generated our algorithm \eqref{algo} to a minimizer of the objective function of minimal norm. By  $\ox_{k}$ we denote the unique solution of the strongly convex minimization problem
\begin{align*}
 \min_{x \in \mathcal{H}} \left( f(x) + \frac{c}{2k^p} \| x \|^2 \right).
\end{align*}
We know, (see for instance \cite{att-com1996}), that $\lim\limits_{k \to +\infty} \ox_{k}=x^\ast$, where $x^\ast = \argmin\limits_{x \in \argmin f} \| x \|$ is the minimal norm element from the set $\argmin f.$ Obviously, $\{x^*\}=\pr_{\argmin f} 0$ and we have the inequality $\| \ox_{k} \| \leq \| x^\ast \|$ (see \cite{BCL}).

Since $\ox_{k}$ is the unique minimum of the strongly convex function $f_k(x)=f(x)+\frac{c}{2k^p}\|x\|^2,$ obviously one has
\begin{equation}\label{fontos0}
\p f_k(\ox_{k})=\p f(\ox_{k})+\frac{c}{k^p}\ox_{k}\ni0.
\end{equation}
Further, Lemma \ref{lfosapen} c) leads to the following.
For every $p_1>p$ there exists $k_0\ge 1$ such that
\begin{equation}\label{lfos}
\left\|\ox_{k+1}-\ox_k\right\|\le\min\left(\frac{p_1}{k}\|\ox_{k}\|,\frac{p_1}{k+1}\|\ox_{k+1}\|\right)\mbox{ for every } k\ge k_0.
\end{equation}

Note that since $f_k$ is strongly convex, from the subgradient inequality we have
\begin{equation}\label{fontos2}
f_k(y)-f_k(x)\ge\<u_k,y-x\>+\frac{c}{2k^p}\|x-y\|^2,\mbox{ for all }x,y\in\mathcal{H}\mbox{ and }u_k\in\p f_k(x).
\end{equation}
In particular
\begin{equation}\label{fontos3}
f_k(x)-f_k(\ox_k)\ge\frac{c}{2k^p}\|x-\ox_k\|^2,\mbox{ for all }x\in\mathcal{H}.
\end{equation}

Finally, observe that for all $x,y\in\mathcal{H}$, one has

\begin{equation}\label{fontos5}
f(x)-f(y)=(f_k(x)-f_k(\ox_k))+(f_k(\ox_k)-f_k(y))+\frac{c}{2k^p}(\|y\|^2-\|x\|^2)\le f_k(x)-f_k(\ox_k)+\frac{c}{2k^p}\|y\|^2.
\end{equation}

\subsection{Convergence rates}

Concerning convergence rates for the function values, discrete velocity and subgradient even for this restrictive case we obtain some results that are comparable to the convergence rates obtained for the famous Nesterov algorithm \cite{Nest1}.

The main result of the present section is the following.

\begin{Theorem}\label{strconvergencerates}
Assume that $0<q< 1$, $1<p\le q+1$, $\l_k=\l k^\d,\,\l>0,\,\d\le 0$  and if $\d=0$ then $\l\in]0,1[.$ Let $(x_k)$ be a sequence generated by \eqref{algo}. For every $k\ge 2$ let us denote by $u_k$ the element  from $\p f(x_k)$ that satisfies \eqref{formx} with equality, i.e.,
$$x_{k}=\a_{k-1}(x_{k-1}-x_{k-2})-\l_{k-1} u_k +\left(1- c_{k-1}\right) x_{k-1}.$$  Then   the following  results are valid.
\begin{enumerate}
\item[(i)] If  $p<q+1$ then $(x_k)$ is bounded and
\[f(x_k)-\min_{\cH} f=\mathcal{O}(k^{-p-\d}),\,\|x_{k}-x_{k-1}\|=\mathcal{O}(k^{-\frac{p}{2}})\mbox{ and }\|u_k\|=\mathcal{O}(k^{-\frac{p}{2}-\d})\mbox{ as } k\to+\infty.\]
Further, for all $s\in\left]\frac12,\frac{p}{2}\right[$ one has
\[\ds\sum_{k=1}^{+\infty} k^{2s+\d-1}(f(x_k)-\min_{\cH} f)<+\infty,\,\ds\sum_{k=1}^{+\infty} k^{2s-q}\|x_k-x_{k-1}\|^2<+\infty\mbox{ and }\ds\sum_{k=2}^{+\infty} k^{2s+2\d}\|u_k\|^2<+\infty.\]
 Moreover, the following ergodic type convergence results hold.
$$\limsup_{n\to+\infty}\frac{\sum_{k=1}^n k^{q+\d}(f_{k-1}(x_{k-1})-f_{k-1}(\ox_{k-1}))}{n^{q+1-p}}<+\infty, \limsup_{n\to+\infty}\frac{\sum_{k=1}^n k\|x_{k}-x_{k-1}\|^2}{n^{q+1-p}}<+\infty$$
$$\mbox{ and } \limsup_{n\to+\infty}\frac{\sum_{k=1}^n k^{q+1+2\d}\|u_k\|^2}{n^{q+1-p}}<+\infty.$$
\item[(ii)] If  $p=q+1$ then
\[f(x_k)-\min_{\cH} f=\mathcal{O}(k^{-p-\d}\ln k),\,\|x_{k}-x_{k-1}\|=\mathcal{O}(k^{-\frac{p}{2}}\sqrt{\ln k})\mbox{ and }\|u_k\|=\mathcal{O}(k^{-\frac{p}{2}-\d}\sqrt{\ln k})\mbox{ as } k\to+\infty.\]
Further, for all $s\in\left]\frac12,\frac{p}{2}\right[$ one has
\[\ds\sum_{k=1}^{+\infty} k^{2s+\d-1}(f(x_k)-\min_{\cH} f)<+\infty,\,\ds\sum_{k=1}^{+\infty} k^{2s-q}\|x_k-x_{k-1}\|^2<+\infty\mbox{ and }\ds\sum_{k=2}^{+\infty} k^{2s+2\d}\|u_k\|^2<+\infty.\]
 Moreover, the following ergodic type convergence results hold.
$$\limsup_{n\to+\infty}\frac{\sum_{k=1}^n k^{q+\d}(f_{k-1}(x_{k-1})-f_{k-1}(\ox_{k-1}))}{\ln n}<+\infty, \limsup_{n\to+\infty}\frac{\sum_{k=1}^n k\|x_{k}-x_{k-1}\|^2}{\ln n}<+\infty$$
$$\mbox{ and } \limsup_{n\to+\infty}\frac{\sum_{k=1}^n k^{q+1+2\d}\|u_k\|^2}{\ln n}<+\infty.$$
Additionally, if $\d<0$ one has
$$\limsup_{n\to+\infty}\frac{\sum_{k=1}^n k^{q+\d}(f(x_{k-1})-\min_{\cH} f)}{\ln n}<+\infty.$$
\end{enumerate}
\end{Theorem}

\begin{proof}
 Consider  first $a_k=ak^{u},\,b_k=k^{v},\, u,v\in\R,\,a>0,\,u+1\ge v\ge u+q$ and define, for every $k \geq 2$, the following discrete energy functional.

\begin{align}\label{Lyapunovstr}
E_k&=\mu_{k-1}(f_{k-1}(x_{k-1})-f_{k-1}(\ox_{k-1}))+ \| a_{k-1}(x_{k-1}-\ox_{k-1}) +b_{k-1} (x_k-x_{k-1}+\l_{k-1}u_k)\|^2\\
\nonumber&+\nu_{k-1}\|x_{k-1}-\ox_{k-1}\|^2+\sigma_{k-1}\|x_{k-1}\|^2,
\end{align}
where the sequences $(\mu_k)$, $(\nu_k)$ and $(\sigma_k)$ will be specified lather.
\vskip0.3cm
{\bf I. Lyapunov analysis}
\vskip0.3cm
All the following estimates hold after an index  $k$ big enough.  Now, if we denote $v_k= \| a_{k-1}(x_{k-1}-\ox_{k-1}) +b_{k-1} (x_k-x_{k-1}+\l_{k-1}u_k)\|^2$ then proceeding as in the proof of Theorem \eqref{convergencealgorithm} we obtain
\begin{align}\label{forEk1str}
v_k=&a_{k-1}b_{k-1}\|x_k-\ox_{k-1}\|^2+(a_{k-1}^2-a_{k-1}b_{k-1})\|x_{k-1}-\ox_{k-1}\|^2+(b_{k-1}^2-a_{k-1}b_{k-1})\|x_k-x_{k-1}\|^2\\
\nonumber&+b_{k-1}^2\l_{k-1}^2\|u_k\|^2+2a_{k-1}b_{k-1}\l_{k-1}\<u_k,x_{k}-\ox_{k-1}\>\\
\nonumber&+(2b_{k-1}^2-2a_{k-1}b_{k-1})\l_{k-1}\<u_k,x_k-x_{k-1}\>.
\end{align}

Further, from (\ref{discreteenergy}) we have
$$v_{k+1}=\|a_{k}(x_{k}-\ox_k)+b_{k}(\a_k(x_k-x_{k-1})-c_k x_k)\|^2.$$
Therefore, after development we get
\begin{align}\label{forEkstr+1}
v_{k+1}=&a_k^2\|x_k-\ox_k\|^2+\a_k^2b_k^2\|x_k-x_{k-1}\|^2+b_k^2c_k^2\|x_k\|^2+2\a_k a_kb_k\<x_k-x_{k-1},x_k-\ox_k\> \\
\nonumber&-2\a_kb_k^2c_k\<x_k-x_{k-1},x_k\>-2a_kb_kc_k\<x_k,x_k-\ox_k\>.
\end{align}
Further,
\begin{eqnarray*}
&&2\a_k a_kb_k\<x_k-x_{k-1},x_k-\ox_k\>=-\a_k a_kb_k(\|x_{k-1}-\ox_k\|-\|x_k-x_{k-1}\|^2-\|x_k-\ox_k\|^2) \\
&&-2\a_kb_k^2c_k\<x_k-x_{k-1},x_k\>=\a_kb_k^2c_k(\|x_{k-1}\|^2-\|x_k-x_{k-1}\|^2-\|x_k\|^2)
\\
&&-2a_kb_kc_k\<x_k,x_k-\ox_k\>=a_kb_kc_k(\|\ox_k\|^2-\|x_k-\ox_k\|^2-\|x_k\|^2).
\end{eqnarray*}
Hence, (\ref{forEkstr+1}) yields
\begin{align}\label{forEkstr+11}
v_{k+1}&=(a_k^2+\a_k a_kb_k-a_kb_kc_k)\|x_k-\ox_k\|^2-\a_k a_kb_k\|x_{k-1}-\ox_k\|^2
\\
\nonumber& +(\a_k^2b_k^2+\a_k a_kb_k-\a_kb_k^2c_k)\|x_k-x_{k-1}\|^2+(b_k^2c_k^2-\a_kb_k^2c_k -a_kb_kc_k)\|x_k\|^2\\
\nonumber&+\a_kb_k^2c_k\|x_{k-1}\|^2+a_kb_kc_k\|\ox_k\|^2.
\end{align}
Consequently, one has
\begin{align}\label{forEkstrdecay}
v_{k+1}-v_k&=(a_k^2+\a_k a_kb_k-a_kb_kc_k)\|x_k-\ox_k\|^2-a_{k-1}b_{k-1}\|x_k-\ox_{k-1}\|^2\\
\nonumber&-\a_k a_kb_k\|x_{k-1}-\ox_k\|^2-(a_{k-1}^2-a_{k-1}b_{k-1})\|x_{k-1}-\ox_{k-1}\|^2
\\
\nonumber& +(\a_k^2b_k^2+\a_k a_kb_k-\a_kb_k^2c_k-b_{k-1}^2+a_{k-1}b_{k-1})\|x_k-x_{k-1}\|^2\\
\nonumber&+(b_k^2c_k^2-\a_kb_k^2c_k -a_kb_kc_k)\|x_k\|^2+\a_kb_k^2c_k\|x_{k-1}\|^2-b_{k-1}^2\l_{k-1}^2\|u_k\|^2\\
\nonumber&+(2b_{k-1}^2-2a_{k-1}b_{k-1})\l_{k-1}\<u_k,x_{k-1}-x_k\>\\
\nonumber&
+2a_{k-1}b_{k-1}\l_{k-1}\<u_k,\ox_{k-1}-x_{k}\>+a_kb_kc_k\|\ox_k\|^2.
\end{align}

Now, by using the sub-gradient inequality we get
\begin{align}\label{fforEK3}
(2b_{k-1}^2&-2a_{k-1}b_{k-1})\l_{k-1}\<u_k,x_{k-1}-x_k\>+2a_{k-1}b_{k-1}\l_{k-1}\<u_k,\ox_{k-1}-x_{k}\>\\
\nonumber\le& -2b_{k-1}^2\l_{k-1}f(x_k)+(2b_{k-1}^2-2a_{k-1}b_{k-1})\l_{k-1}f(x_{k-1})+2a_{k-1}b_{k-1}\l_{k-1}f(\ox_{k-1})\\
\nonumber=& -2b_{k-1}^2\l_{k-1}(f_k(x_k)-f_k(\ox_k))+2b_{k-2}^2\l_{k-2}(f_{k-1}(x_{k-1})-f_{k-1}(\ox_{k-1}))\\
\nonumber&+[(2b_{k-1}^2-2a_{k-1}b_{k-1})\l_{k-1}-2b_{k-2}^2\l_{k-2}](f_{k-1}(x_{k-1})-f_{k-1}(\ox_{k-1}))\\
\nonumber&+2b_{k-1}^2\l_{k-1}(f_{k-1}(\ox_{k-1})-f_k(\ox_{k}))\\
\nonumber&+b_{k-1}^2\l_{k-1}c_k\|x_k\|^2+(a_{k-1}b_{k-1}\l_{k-1}c_{k-1}-b_{k-1}^2\l_{k-1}c_{k-1})\|x_{k-1}\|^2\\
\nonumber&-a_{k-1}b_{k-1}\l_{k-1}c_{k-1}\|\ox_{k-1}\|^2.
\end{align}
Further, according to \eqref{fontos3} one has $f_{k-1}(\ox_{k})-f_{k-1}(\ox_{k-1})\ge\frac{c_{k-1}}{2}\|\ox_k-\ox_{k-1}\|^2$ hence
\begin{align*}
2b_{k-1}^2\l_{k-1}(f_{k-1}(\ox_{k-1})-f_k(\ox_{k}))&=2b_{k-1}^2\l_{k-1}\left(f_{k-1}(\ox_{k-1})-f_{k-1}(\ox_{k})+\frac{c_{k-1}-c_k}{2}\|\ox_{k}\|^2\right)\\
&\le 2b_{k-1}^2\l_{k-1}\left(-\frac{c_{k-1}}{2}\|\ox_k-\ox_{k-1}\|^2+\frac{c_{k-1}-c_k}{2}\|\ox_{k}\|^2\right)
\end{align*}
hence \eqref{fforEK3} becomes
\begin{align}\label{fforEK4}
(2b_{k-1}^2&-2a_{k-1}b_{k-1})\l_{k-1}\<u_k,x_{k-1}-x_k\>+2a_{k-1}b_{k-1}\l_{k-1}\<u_k,\ox_{k-1}-x_{k}\>\\
\nonumber\le&  -2b_{k-1}^2\l_{k-1}(f_k(x_k)-f_k(\ox_k))+2b_{k-2}^2\l_{k-2}(f_{k-1}(x_{k-1})-f_{k-1}(\ox_{k-1}))\\
\nonumber&+[(2b_{k-1}^2-2a_{k-1}b_{k-1})\l_{k-1}-2b_{k-2}^2\l_{k-2}](f_{k-1}(x_{k-1})-f_{k-1}(\ox_{k-1}))\\
\nonumber&+b_{k-1}^2\l_{k-1}c_k\|x_k\|^2+(a_{k-1}b_{k-1}\l_{k-1}c_{k-1}-b_{k-1}^2\l_{k-1}c_{k-1})\|x_{k-1}\|^2\\
\nonumber&+b_{k-1}^2\l_{k-1}(c_{k-1}-c_k)\|\ox_{k}\|^2-a_{k-1}b_{k-1}\l_{k-1}c_{k-1}\|\ox_{k-1}\|^2-b_{k-1}^2\l_{k-1}c_{k-1}\|\ox_k-\ox_{k-1}\|^2.
\end{align}

Combining \eqref{forEkstrdecay} and \eqref{fforEK4} we get
\begin{align}\label{forEkstrdecay1}
&v_{k+1}-v_k   +2b_{k-1}^2\l_{k-1}(f_k(x_k)-f_k(\ox_k))-2b_{k-2}^2\l_{k-2}(f_{k-1}(x_{k-1})-f_{k-1}(\ox_{k-1}))\\
\nonumber&-[(2b_{k-1}^2-2a_{k-1}b_{k-1})\l_{k-1}-2b_{k-2}^2\l_{k-2}](f_{k-1}(x_{k-1})-f_{k-1}(\ox_{k-1}))+b_{k-1}^2\l_{k-1}^2\|u_k\|^2\\
\nonumber&\le (a_k^2+\a_k a_kb_k-a_kb_kc_k)\|x_k-\ox_k\|^2-a_{k-1}b_{k-1}\|x_k-\ox_{k-1}\|^2\\
\nonumber&-\a_k a_kb_k\|x_{k-1}-\ox_k\|^2-(a_{k-1}^2-a_{k-1}b_{k-1})\|x_{k-1}-\ox_{k-1}\|^2\\
\nonumber& +(\a_k^2b_k^2+\a_k a_kb_k-\a_kb_k^2c_k-b_{k-1}^2+a_{k-1}b_{k-1})\|x_k-x_{k-1}\|^2\\
\nonumber&+(b_k^2c_k^2-\a_kb_k^2c_k -a_kb_kc_k+b_{k-1}^2\l_{k-1}c_k)\|x_k\|^2\\
\nonumber&+(a_{k-1}b_{k-1}\l_{k-1}c_{k-1}-b_{k-1}^2\l_{k-1}c_{k-1}+\a_kb_k^2c_k)\|x_{k-1}\|^2\\
\nonumber&+[b_{k-1}^2\l_{k-1}(c_{k-1}-c_k)+a_kb_kc_k]\|\ox_{k}\|^2-a_{k-1}b_{k-1}\l_{k-1}c_{k-1}\|\ox_{k-1}\|^2\\
\nonumber&-b_{k-1}^2\l_{k-1}c_{k-1}\|\ox_k-\ox_{k-1}\|^2.
\end{align}

We estimate in what follows the entities $-a_{k-1}b_{k-1}\|x_k-\ox_{k-1}\|^2$ and $-\a_k a_kb_k\|x_{k-1}-\ox_k\|^2.$
Using the straightforward inequality $\pm 2\<a,b\>\le \frac{1}{s}\|a\|^2+s\|b\|^2$ for all $s>0$  we obtain that
\begin{align}\label{interm1}
-a_{k-1}b_{k-1}\|x_k-\ox_{k-1}\|^2&=-a_{k-1}b_{k-1}\|(x_k-\ox_k)+(\ox_k-\ox_{k-1})\|^2=-a_{k-1}b_{k-1}\|x_k-\ox_{k}\|^2\\
\nonumber&-a_{k-1}b_{k-1}\|\ox_k-\ox_{k-1}\|^2-2a_{k-1}b_{k-1}\<x_k-x_{k-1},\ox_k-\ox_{k-1}\>\\
\nonumber&-2a_{k-1}b_{k-1}\<x_{k-1}-\ox_{k-1},\ox_k-\ox_{k-1}\>+2a_{k-1}b_{k-1}\<\ox_k-\ox_{k-1},\ox_k-\ox_{k-1}\>\\
\nonumber&\le -a_{k-1}b_{k-1}\|x_k-\ox_{k}\|^2+2 a_{k-1}b_{k-1}\|x_k-x_{k-1}\|^2\\
\nonumber&+ \left(1+\frac{1}{2}\right)a_{k-1}b_{k-1}\|\ox_k-\ox_{k-1}\|^2+2a_{k-1}b_{k-1}\<x_{k-1}-\ox_{k-1},\ox_{k-1}-\ox_k\>.
\end{align}
Further,
\begin{align}\label{interm2}
-\a_k a_kb_k\|x_{k-1}-\ox_k\|^2&=-\a_k a_kb_k\|x_{k-1}-\ox_{k-1}\|^2-\a_k a_kb_k\|\ox_{k-1}-\ox_k\|^2\\
\nonumber&-2\a_k a_kb_k\<x_{k-1}-\ox_{k-1},\ox_{k-1}-\ox_k\>,
\end{align}
and for $s_{k-1}=\frac{s}{(k-1)^{p-q}}$ with $s<\frac{c}{\a}$  one has
\begin{align}\label{interm3}
(2a_{k-1}b_{k-1}-2\a_k a_kb_k)&\<x_{k-1}-\ox_{k-1},\ox_{k-1}-\ox_k\>\le\\
\nonumber&(a_{k-1}b_{k-1}-\a_k a_kb_k)\left(s_{k-1}\|x_{k-1}-\ox_{k-1}\|^2+\frac{1}{s_{k-1}}\|\ox_{k-1}-\ox_k\|^2\right)
\end{align}

Now, combining  \eqref{interm1} and \eqref{interm2} and \eqref{interm3} it holds
\begin{align}\label{interm4}
&-a_{k-1}b_{k-1}\|x_k-\ox_{k-1}\|^2-\a_k a_kb_k\|x_{k-1}-\ox_k\|^2\le-a_{k-1}b_{k-1}\|x_k-\ox_{k}\|^2\\
\nonumber&+2a_{k-1}b_{k-1}\|x_k-x_{k-1}\|^2+(-\a_k a_kb_k+(a_{k-1}b_{k-1}-\a_k a_kb_k)s_{k-1})\|x_{k-1}-\ox_{k-1}\|^2\\
\nonumber&+\left(\left(1+\frac{1}{s_{k-1}}\right)(a_{k-1}b_{k-1}-\a_k a_kb_k)+\frac{a_{k-1}b_{k-1}}{2}\right)\|\ox_{k}-\ox_{k-1}\|^2.
\end{align}
Injecting \eqref{interm4} in \eqref{forEkstrdecay1} we get

\begin{align}\label{forEkstrdecay11}
&v_{k+1}-v_k   +2b_{k-1}^2\l_{k-1}(f_k(x_k)-f_k(\ox_k))-2b_{k-2}^2\l_{k-2}(f_{k-1}(x_{k-1})-f_{k-1}(\ox_{k-1}))\\
\nonumber&-[(2b_{k-1}^2-2a_{k-1}b_{k-1})\l_{k-1}-2b_{k-2}^2\l_{k-2}](f_{k-1}(x_{k-1})-f_{k-1}(\ox_{k-1}))+b_{k-1}^2\l_{k-1}^2\|u_k\|^2\\
\nonumber&\le (a_k^2+\a_k a_kb_k-a_kb_kc_k-a_{k-1}b_{k-1})\|x_k-\ox_{k}\|^2\\
\nonumber&+(-a_{k-1}^2+(1+s_{k-1})(a_{k-1}b_{k-1}-\a_k a_kb_k))\|x_{k-1}-\ox_{k-1}\|^2\\
\nonumber& +(\a_k^2b_k^2+\a_k a_kb_k-\a_kb_k^2c_k-b_{k-1}^2+3a_{k-1}b_{k-1})\|x_k-x_{k-1}\|^2\\
\nonumber&+(b_k^2c_k^2-\a_kb_k^2c_k -a_kb_kc_k+b_{k-1}^2\l_{k-1}c_k)\|x_k\|^2\\
\nonumber&+(a_{k-1}b_{k-1}\l_{k-1}c_{k-1}-b_{k-1}^2\l_{k-1}c_{k-1}+\a_kb_k^2c_k)\|x_{k-1}\|^2\\
\nonumber&+\left(b_{k-1}^2\l_{k-1}(c_{k-1}-c_k)+a_kb_kc_k\right)\|\ox_{k}\|^2-a_{k-1}b_{k-1}\l_{k-1}c_{k-1}\|\ox_{k-1}\|^2\\
\nonumber&+\left(\left(1+\frac{1}{s_{k-1}}\right)(a_{k-1}b_{k-1}-\a_k a_kb_k)+\frac{a_{k-1}b_{k-1}}{2}-b_{k-1}^2\l_{k-1}c_{k-1}\right)\|\ox_{k}-\ox_{k-1}\|^2.
\end{align}

Consider now $u=r-1,\,v=r$ and assume that $ a>1+q,\,r\in\left(\frac12, \frac{q+1}{2}\right]$.
Further, let
  $\mu_k=2b_{k-1}^2\l_{k-1}$, $\nu_k=-a_k^2-\a_k a_kb_k+a_kb_kc_k+a_{k-1}b_{k-1}$ and $\sigma_k=-b_k^2c_k^2+\a_kb_k^2c_k +a_kb_kc_k-b_{k-1}^2\l_{k-1}c_k$ for all $k\ge 1.$

Next  we show that all the sequences defined above are positive after an index $K_0$ big enough. For an easier readability we emphasize that by $h_k-\mathcal{O}(k^l)$ we understand the difference of a sequence $h_k$ and a positive sequence of order $\mathcal{O}(k^l)$ as $k\to+\infty.$ Similarly, by $h_k+\mathcal{O}(k^l)$ we understand the sum of a sequence $h_k$ and a positive sequence of order $\mathcal{O}(k^l)$ as $k\to+\infty.$  Further, by $s\mathcal{O}(k^l),\,s>0$ we understand the positive sequences $u_k$ that after an index satisfy $u_k\le s k^l.$ All the estimates bellow hold after an index $K_0$ big enough.

Obviously,  one has
\begin{align}\label{ordmu}
&\mu_k=2\l(k-1)^{2r+\d}>0\mbox{ and }\mu_k=\mathcal{O}(k^{2r+\d}).
\end{align}

Further, since $q<1<p$  one has
\begin{align}\label{ordnu}
\nu_k&=-a^2k^{2r-2}-\left(1-\frac{\a}{k^q}\right) ak^{2r-1}+ack^{2r-1-p}+a(k-1)^{2r-1}\\
\nonumber&=a\a k^{2r-1-q}-\mathcal{O}(k^{2r-2})+\mathcal{O}(k^{2r-1-p})>0\mbox{ and }\nu_k=\mathcal{O}(k^{2r-1-q}).
\end{align}

Now, since $\l_k=\l k^\d<1$, for $k$ big enough, i.e. $\d\le 0$ and $0<\l<1$ if $\d=0$,  one has
\begin{align}\label{ordsigma}
\sigma_k&=-c^2k^{2r-2p}+\left(1-\frac{\a}{k^q}\right)ck^{2r-p} +ack^{2r-1-p}-\l c(k-1)^{2r+\d}k^{-p}\\
\nonumber&=ck^{-p}(k^{2r}-\l(k-1)^{2r+\d})+ack^{2r-1-p} -\a ck^{2r-q-p}-c^2k^{2r-2p}\\
\nonumber&=ck^{-p}(k^{2r}-\l(k-1)^{2r+\d})+\mathcal{O}(k^{2r-1-p})-\mathcal{O}(k^{2r-q-p})-\mathcal{O}(k^{2r-2p})>0\mbox{ and }\s_k=\mathcal{O}(k^{2r-p}).
\end{align}

Consequently, $E_k\ge 0$ for all $k\ge K_0.$

In other words \eqref{forEkstrdecay11} can be written as
\begin{align}\label{forEkstrdecay111}
&E_{k+1}-E_k +(-\a_k^2b_k^2-\a_k a_kb_k+\a_kb_k^2c_k+b_{k-1}^2-3a_{k-1}b_{k-1})\|x_k-x_{k-1}\|^2 \\
\nonumber&+b_{k-1}^2\l_{k-1}^2\|u_k\|^2+(2a_{k-1}b_{k-1}\l_{k-1}+2b_{k-2}^2\l_{k-2}-2b_{k-1}^2\l_{k-1})(f_{k-1}(x_{k-1})-f_{k-1}(\ox_{k-1}))\\
\nonumber&+(-\a_{k-1} a_{k-1}b_{k-1}+a_{k-1}b_{k-1}c_{k-1}+a_{k-2}b_{k-2}-(1+s_{k-1})(a_{k-1}b_{k-1}-\a_k a_kb_k))\|x_{k-1}-\ox_{k-1}\|^2\\
\nonumber&+(-b_{k-1}^2c_{k-1}^2+\a_{k-1}b_{k-1}^2c_{k-1} +a_{k-1}b_{k-1}c_{k-1}-b_{k-2}^2\l_{k-2}c_{k-1}-a_{k-1}b_{k-1}\l_{k-1}c_{k-1}+\\
\nonumber&+b_{k-1}^2\l_{k-1}c_{k-1}-\a_kb_k^2c_k)\|x_{k-1}\|^2\\
\nonumber&\le\left(b_{k-1}^2\l_{k-1}(c_{k-1}-c_k)+a_kb_kc_k\right)\|\ox_{k}\|^2-a_{k-1}b_{k-1}\l_{k-1}c_{k-1}\|\ox_{k-1}\|^2\\
\nonumber&+\left(\left(1+\frac{1}{s_{k-1}}\right)(a_{k-1}b_{k-1}-\a_k a_kb_k)+\frac{a_{k-1}b_{k-1}}{2}-b_{k-1}^2\l_{k-1}c_{k-1}\right)\|\ox_{k}-\ox_{k-1}\|^2.
\end{align}

For simplicity, let us denote
\begin{align*}
&\xi_k=b_{k-1}^2\l_{k-1}^2\\
&m_k=2a_{k-1}b_{k-1}\l_{k-1}+2b_{k-2}^2\l_{k-2}-2b_{k-1}^2\l_{k-1}\\
&n_k=-\a_{k-1} a_{k-1}b_{k-1}+a_{k-1}b_{k-1}c_{k-1}+a_{k-2}b_{k-2}-(1+s_{k-1})(a_{k-1}b_{k-1}-\a_k a_kb_k)\\
&\eta_k=-\a_k^2b_k^2-\a_k a_kb_k+\a_kb_k^2c_k+b_{k-1}^2-3a_{k-1}b_{k-1}\\
&t_k=-b_{k-1}^2c_{k-1}^2+\a_{k-1}b_{k-1}^2c_{k-1} +a_{k-1}b_{k-1}c_{k-1}-b_{k-2}^2\l_{k-2}c_{k-1}-a_{k-1}b_{k-1}\l_{k-1}c_{k-1}+\\
&+b_{k-1}^2\l_{k-1}c_{k-1}-\a_kb_k^2c_k,
\end{align*}
and we show that all the sequences above are positive after an index $K_1\ge K_0$ big enough.

First one has
\begin{align}\label{ordxi}
\xi_k&=\l^2(k-1)^{2r+2\d}>0\mbox{ and }\xi_k=\mathcal{O}(k^{2r+2\d}).
\end{align}

Obviously, since  $a>1+q\ge 2r$ one has

\begin{align}\label{ordm}
m_k&=2a\l(k-1)^{2r-1+\d}+2\l((k-2)^{2r+\d}-(k-1)^{2r+\d})\\
\nonumber&=2a\l(k-1)^{2r-1+\d}-2\l(2r+\d)\mathcal{O}(k^{2r-1+\d})>0\mbox{ and }m_k=\mathcal{O}(k^{2r-1+\d}).
\end{align}
If $q<1$, $1+q>p>1$, by taking into account that  $(k-1)^{2r-1-q}-k^{2r-1-q}=0$ if $r=\frac{q+1}{2}$ and
 $(k-1)^{2r-1-q}-k^{2r-1-q}=\mathcal{O}(k^{2r-2-q})$ if $r<\frac{q+1}{2}$ and  $s<\frac{c}{\a}$ one has
\begin{align}\label{ordn}
n_k&=-\left(1-\frac{\a}{(k-1)^q}\right) a(k-1)^{2r-1}+ac(k-1)^{2r-1-p}+a(k-2)^{2r-1}\\
\nonumber&-\left(1+\frac{s}{(k-1)^{p-q}}\right)\left(a(k-1)^{2r-1}-\left(1-\frac{\a}{k^q}\right)ak^{2r-1}\right)\\
\nonumber&=ac(k-1)^{2r-1-p}+a((k-2)^{2r-1}+k^{2r-1}-2(k-1)^{2r-1})\\
\nonumber&-\frac{as}{(k-1)^{p-q}}((k-1)^{2r-1}-k^{2r-1}+\a k^{2r-1-q})+a\a((k-1)^{2r-1-q}-k^{2r-1-q})\\
\nonumber&=ac(k-1)^{2r-1-p}-\mathcal{O}(k^{2r-3})-as\a\mathcal{O}(k^{2r-1-p})+a \a\mathcal{O}((k-1)^{2r-1-q}-k^{2r-1-q})>0\\
\nonumber&\mbox{ and }n_k=\mathcal{O}(k^{2r-1-p}).
\end{align}
If $q<1$, $1+q=p$ then  $s_k=\frac{s}{k}$ and by taking into account that  $(k-1)^{2r-1-q}-k^{2r-1-q}=0,$ if $r=\frac{q+1}{2}$ and
 $(k-1)^{2r-1-q}-k^{2r-1-q}=(1+q-2r)\mathcal{O}(k^{2r-2-q})$ if $r<\frac{q+1}{2}$ and $s<\frac{c}{\a}$ one has
\begin{align}\label{ordnpq1}
n_k&=ac(k-1)^{2r-2-q}+a((k-2)^{2r-1}+k^{2r-1}-2(k-1)^{2r-1})\\
\nonumber&-\frac{as}{k-1}((k-1)^{2r-1}-k^{2r-1}+\a k^{2r-1-q})+a\a((k-1)^{2r-1-q}-k^{2r-1-q})\\
\nonumber&=ac(k-1)^{2r-2-q}-\frac{as\a}{k-1}k^{2r-1-q}+\mathcal{O}(k^{2r-3})-\mathcal{O}(k^{2r-3})\\
\nonumber&+a \a\mathcal{O}((k-1)^{2r-1-q}-k^{2r-1-q})>0\mbox{ and }n_k=\mathcal{O}(k^{2r-2-q}).
\end{align}

Concerning $\eta_k$, since  $p>1>q$ one has
\begin{align}\label{ordeta}
\eta_k&=-\left(1-\frac{\a}{k^q}\right)^2 k^{2r}-\left(1-\frac{\a}{k^q}\right)ak^{2r-1}+\left(1-\frac{\a}{k^q}\right)ck^{2r-p}+(k-1)^{2r}-3a(k-1)^{2r-1}\\
\nonumber&=2\a k^{2r-q}+((k-1)^{2r}-k^{2r})-\a^2k^{2r-2q} -ak^{2r-1}+\a ak^{2r-1-q}+\left(1-\frac{\a}{k^q}\right)ck^{2r-p}\\
\nonumber&-3a(k-1)^{2r-1}=2\a k^{2r-q}-\mathcal{O}(k^{2r-1})>0\mbox{ and }\eta_k=\mathcal{O}(k^{2r-q}).
\end{align}
Now, since $\l_k=\l k^\d\le 1$, for $k$ big enough, i.e. $\d\le 0$ and $0<\l< 1$ if $\d=0$, further $a>1+q\ge 2r$, hence $a>|2r-p|$ if $\d<0$ and if $\d=0$ then $a>\frac{(2r-p)-2\l r}{1-\l}$, one has
\begin{align}\label{ordt}
t_k&=-c^2(k-1)^{2r-2p}+\left(1-\frac{\a}{(k-1)^q}\right)c(k-1)^{2r-p}+ac(k-1)^{2r-1-p}-\l c(k-2)^{2r+\d}(k-1)^{-p}\\
\nonumber&-a\l c(k-1)^{2r-1+\d-p}+\l c(k-1)^{2r+\d-p}-\left(1-\frac{\a}{k^q}\right)ck^{2r-p}\\
\nonumber&=(ac(k-1)^{2r-1-p}-a\l c(k-1)^{2r-1+\d-p})+c((k-1)^{2r-p}-k^{2r-p})\\
\nonumber&+\l c(k-1)^{-p}((k-1)^{2r+\d}-(k-2)^{2r+\d})+\a c(k^{2r-q-p}-(k-1)^{2r-q-p})-c^2(k-1)^{2r-2p}\\
\nonumber&=(ac(k-1)^{2r-1-p}-a\l c(k-1)^{2r-1+\d-p})-c(2r-p)\mathcal{O}(k^{2r-1-p})+\l c(2r+\d)\mathcal{O}(k^{2r-1-p+\d})\\
\nonumber&-\mathcal{O}(k^{2r-q-1-p})-\mathcal{O}(k^{2r-2p})>0\mbox{ and }t_k=\mathcal{O}(k^{2r-1-p}).
\end{align}
\vskip0.3cm

Concerning the right hand side of \eqref{forEkstrdecay111}, in what follows we show that
$$\sum_{k=1}^{+\infty}\left(\left(1+\frac{1}{s_{k-1}}\right)(a_{k-1}b_{k-1}-\a_k a_kb_k)+\frac{a_{k-1}b_{k-1}}{2}\right)\|\ox_{k}-\ox_{k-1}\|^2< +\infty.$$
Let us denote
\begin{equation}\label{Sk}
S_k:=\left(\left(1+\frac{1}{s_{k-1}}\right)(a_{k-1}b_{k-1}-\a_k a_kb_k)+\frac{a_{k-1}b_{k-1}}{2}\right)\|\ox_{k}-\ox_{k-1}\|^2.
\end{equation}

Note that according to  \eqref{lfos} one has $\left\|\ox_{k}-\ox_{k-1}\right\|\le\frac{p_1}{k}\|\ox_{k}\|$ for some $p_1>p$ and all $k$ big enough. Further $\|\ox_{k}\|^2\le \|x^*\|^2$, hence we have
$$\left\|\ox_{k}-\ox_{k-1}\right\|^2\le\frac{p_1^2}{k^2}\|x^*\|^2.$$

Therefore, it is enough to show that  $\left(1+\frac{1}{s_{k-1}}\right)(a_{k-1}b_{k-1}-\a_k a_kb_k)+\frac{a_{k-1}b_{k-1}}{2}=\mathcal{O}(k^l)$ as $k\to+\infty$, with $l<1.$

Indeed,
\begin{align*}
&\left(1+\frac{1}{s_{k-1}}\right)(a_{k-1}b_{k-1}-\a_k a_kb_k)+\frac{a_{k-1}b_{k-1}}{2}=\left(1+\frac{(k-1)^{p-q}}{s}\right)\left(a(k-1)^{2r-1}-\left(1-\frac{\a}{k^q}\right)ak^{2r-1}\right)\\
&+\frac{a}{2}(k-1)^{2r-1}\le C_1(k-1)^{\max(2r-1-2q+p,2r-1)}.
\end{align*}
Observe that by assumption $q<1$ and $2r\le q+1$ if $p<q+1$, hence  one can take $l=\max(2r-1-q+s,2r-1)<1$ and we obtain that $(S_k)$ is summable.

Further,  for $p=q+1$  if $2r<q+1$ we obtain that $l=\max(2r-1-2q+p,2r-1)<1$, so also in this case $(S_k)$ is summable.

However, in case $p=q+1$ and $2r=q+1$ one has $l=1$, hence $S_k=\mathcal{O}(k^{-1}).$

Now, since $\|\ox_{k}\|^2\le\|x^*\|^2$ and
\begin{align*}
b_{k-1}^2\l_{k-1}(c_{k-1}-c_k)+a_kb_kc_k&=c\l(k-1)^{2r+\d}((k-1)^{-p}-k^{-p})+ack^{2r-1-p}\\
&=\mathcal{O}(k^{2r-1-p}),
\end{align*}
the right hand side of \eqref{forEkstrdecayf} leads to
\begin{align*}
  &\left(b_{k-1}^2\l_{k-1}(c_{k-1}-c_k)+a_kb_kc_k\right)\|\ox_{k}\|^2-a_{k-1}b_{k-1}\l_{k-1}c_{k-1}\|\ox_{k-1}\|^2\\
  &-b_{k-1}^2\l_{k-1}c_{k-1}\|\ox_k-\ox_{k-1}\|^2+S_k\le C_2 k^{2r-1-p}+S_k\mbox{ for some }C_2>0.
\end{align*}

Consequently, \eqref{forEkstrdecay111} leads to

\begin{align}\label{forEkstrdecayf}
&E_{k+1}-E_k +\xi_k\|u_k\|^2+m_k(f_{k-1}(x_{k-1})-f_{k-1}(\ox_{k-1}))+n_k\|x_{k-1}-\ox_{k-1}\|^2\\
\nonumber& +\eta_k\|x_k-x_{k-1}\|^2+t_k\|x_{k-1}\|^2 \le C_2k^{2r-1-p}+S_k\mbox{ for all }k\ge K_1.
\end{align}

Summing up \eqref{forEkstrdecayf} from $k=K_1$ to $k=n\ge K_1$ we obtain
\begin{align}\label{forEkstrdecayfi}
&E_{n+1}  +\sum_{k=K_1}^n \xi_k\|u_k\|^2+\sum_{k=K_1}^n m_k(f_{k-1}(x_{k-1})-f_{k-1}(\ox_{k-1}))+\sum_{k=K_1}^n n_k\|x_{k-1}-\ox_{k-1}\|^2\\
\nonumber& +\sum_{k=K_1}^n \eta_k\|x_k-x_{k-1}\|^2+\sum_{k=K_1}^nt_k\|x_{k-1}\|^2\le C_2\sum_{k=K_1}^n k^{2r-1-p}+\sum_{k=K_1}^n S_k+E_{K_1}\\
\nonumber&\le C_2\sum_{k=K_1}^n k^{2r-1-p}+C\mbox{ for some }C>0.
\end{align}
\vskip0.3cm
{\bf II. Rates}
\vskip0.3cm
In what follows $x^*$ denotes the element of minimum norm from the set $\argmin f.$

We treat first the case $p<q+1.$

Now, if $2r-1-p>-1$, that is $r\in\left(\frac{p}{2},\frac{q+1}{2}\right],$ it is obvious that $\sum_{k=K_1}^n k^{2r-1-p}\to+\infty\mbox{ as }n\to+\infty.$ However, easily can by seen that
$\sum_{k=K_1}^n k^{2r-1-p}=\mathcal{O}(n^{2r-p}).$

Hence, dividing \eqref{forEkstrdecayfi} with $n^{2r-p}$ we obtain at once that there exists $L>0$ such that $\frac{E_{n+1}}{n^{2r-p}}<L$, consequently
$$\frac{\mu_n}{n^{2r-p}}(f_n(x_n)-f_n(\ox_n))\le L\mbox{ and }\frac{\s_n}{n^{2r-p}}\|x_n\|^2\le L\mbox{ for all }n\ge K_1.$$
But according to \eqref{ordsigma} one has $\s_n=\mathcal{O}(n^{2r-p})$  consequently $(x_n)$ is bounded.

From \eqref{ordmu} we have $\mu_n=\mathcal{O}(n^{2r+\d})$,  hence
$$f_n(x_n)-f_n(\ox_n)=\mathcal{O}(n^{-p-\d}).$$
Consequently, for every $\rho<p+\d-1$ one has
$$\sum_{k=1}^{+\infty}k^\rho(f_k(x_k)-f_x(\ox_k))<+\infty.$$

Now, according to \eqref{fontos5} one has
$f(x_n)-f(x^*)\le f_n(x_n)-f_n(\ox_n)+\frac{c}{2n^p}\|x^*\|^2$ hence, since $\d\le0$ we obtain
$$f(x_n)-f(x^*)=\mathcal{O}(n^{-p-\d}).$$

Further, one has $\frac{v_{n+1}}{n^{2r-p}}<L$, hence
$$\frac{\| a_{n}(x_{n}-\ox_n)+b_{n}(\a_n(x_n-x_{n-1})-c_n x_n)\|^2}{n^{2r-p}}<L\mbox{ for all }n\ge K_1.$$
Consequently, $\|an^{\frac{p}{2}-1}(x_{n}-\ox_{n}) +n^{\frac{p}{2}} (\a_n(x_n-x_{n-1})-cn^{-p} x_n)\|^2$ is bounded.
But $(x_n)$ is bounded and $p<2$, hence $an^{\frac{p}{2}-1}(x_{n}-\ox_{n})\to0\mbox{ as }n\to+\infty$ and $-cn^{-\frac{p}{2}} x_n\to0\mbox{ as }n\to+\infty$, consequently $\|n^{\frac{p}{2}} \a_n(x_n-x_{n-1})\|^2$ is bounded. In other words
$$\|x_n-x_{n-1}\|^2=\mathcal{O}(n^{-p}).$$
Hence, for every $\rho<p-1$ one has
$$\sum_{k=1}^{+\infty}k^\rho\|x_k-x_{k-1}\|^2<+\infty.$$

Now, using the definition of $u_n$ we have $\l_{n-1} u_n =(x_{n}-x_{n-1})-\a_{n-1}(x_{n-1}-x_{n-2})+c_{n-1} x_{n-1}$ hence
$$\|\l(n-1)^\d u_n\|\le\|x_{n}-x_{n-1}\|+\a_{n-1}\|x_{n-1}-x_{n-2}\|+c_{n-1}\| x_{n-1}\|=\mathcal{O}(n^{-\frac{p}{2}}).$$

Consequently, $\|u_n\|^2=\mathcal{O}(n^{-p-2\d})$ and for every $\rho<p+2\d-1$ one has
$$\sum_{k=1}^{+\infty}k^\rho\|u_k\|^2<+\infty.$$

Further, by taking $r=\frac{q+1}{2}$ we obtain the following ergodic convergence results.

$$\limsup_{n\to+\infty}\frac{\sum_{k=1}^n m_k(f_{k-1}(x_{k-1})-f_{k-1}(\ox_{k-1}))}{n^{q+1-p}}<+\infty.$$
But according to \eqref{ordm} we have $m_k=\mathcal{O}(k^{q+\d})$, hence
$$\limsup_{n\to+\infty}\frac{\sum_{k=1}^n k^{q+\d}(f_{k-1}(x_{k-1})-f_{k-1}(\ox_{k-1}))}{n^{q+1-p}}<+\infty.$$
Similarly, according to \eqref{ordeta} one has $\eta_k=\mathcal{O}(k^{1})$, hence
$$\limsup_{n\to+\infty}\frac{\sum_{k=1}^n k\|x_{k}-x_{k-1}\|^2}{n^{q+1-p}}<+\infty.$$
Finally, according to \eqref{ordxi} one has $\xi_k=\mathcal{O}(k^{q+1+2\d})$, hence
$$\limsup_{n\to+\infty}\frac{\sum_{k=1}^n k^{q+1+2\d}\|u_k\|^2}{n^{q+1-p}}<+\infty.$$
\vskip0.3cm
Now, if $2r-1-p<-1$, that is $r\in\left(\frac12,\frac{p}{2}\right),$ then  the right hand side of \eqref{forEkstrdecayfi} is finite, hence there exists $C_3>0$ such that
\begin{align}\label{forEkstrdecayfi1}
&E_{n+1}  +\sum_{k=K_1}^n \xi_k\|u_k\|^2+\sum_{k=K_1}^n m_k(f_{k-1}(x_{k-1})-f_{k-1}(\ox_{k-1}))+\sum_{k=K_1}^n n_k\|x_{k-1}-\ox_{k-1}\|^2\\
\nonumber& +\sum_{k=K_1}^n \eta_k\|x_k-x_{k-1}\|^2+\sum_{k=K_1}^nt_k\|x_{k-1}\|^2\le C_2\sum_{k=K_1}^n k^{2r-1-p}+\sum_{k=K_1}^n S_k+E_{K_1}\le C_3.
\end{align}
From \eqref{forEkstrdecayfi1} by using \eqref{ordxi}, \eqref{ordm} and \eqref{ordeta} we obtain the estimates
$$\sum_{k=1}^{+\infty} k^{2r+2\d}\|u_k\|^2<+\infty,$$
$$\sum_{k=1}^{+\infty} k^{2r-1+\d}(f_{k-1}(x_{k-1})-f_{k-1}(\ox_{k-1}))<+\infty$$
and
$$\sum_{k=1}^{+\infty} k^{2r-q}\|x_{k}-x_{k-1}\|^2<+\infty.$$
But  according to \eqref{fontos5} we have
$f(x_{k-1})-f(x^*)\le f_{k-1}(x_{k-1})-f_{k-1}(\ox_{k-1})+\frac{c}{2k^p}\|x^*\|^2.$

Further $\sum_{k=1}^{+\infty} k^{2r-1+\d}\frac{c}{2 k^p}\|x^*\|^2<+\infty$ therefore
$$\sum_{k=1}^{+\infty} k^{2r-1+\d}(f(x_{k-1})-\min_{\cH}f)<+\infty.$$
\vskip0.5cm

In case $p=q+1$ we have seen earlier, that $S_k$ defined by \eqref{Sk} is summable provided  $2r<q+1.$ Further, for $2r=q+1$ one has $S_k=\mathcal{O}(k^{-1}).$ Consequently, the right hand side of \eqref{forEkstrdecayfi}, that is $C_2\sum_{k=K_1}^n k^{2r-1-p}+\sum_{k=K_1}^n S_k+E_{K_1}$ is finite for $2r<q+1$ and is of order $\mathcal{O}(k^{-1})$ for $2r=q+1.$

So assume first that $r\in\left(\frac12,\frac{q+1}{2}\right).$ Then  \eqref{forEkstrdecayfi} becomes:
\begin{align}\label{forEkstrdecayfipq1}
&E_{n+1}  +\sum_{k=K_1}^n \xi_k\|u_k\|^2+\sum_{k=K_1}^n m_k(f_{k-1}(x_{k-1})-f_{k-1}(\ox_{k-1}))+\sum_{k=K_1}^n n_k\|x_{k-1}-\ox_{k-1}\|^2\\
\nonumber& +\sum_{k=K_1}^n \eta_k\|x_k-x_{k-1}\|^2+\sum_{k=K_1}^nt_k\|x_{k-1}\|^2\le C,\mbox{ for some }C>0.
\end{align}
From \eqref{forEkstrdecayfipq1}, for all $r\in\left(\frac12,\frac{q+1}{2}\right)$ we obtain at once the following estimates:

$\sum_{k=1}^{+\infty}k^{2r+2\d}\|u_k\|^2<+\infty,$ $\sum_{k=1}^{+\infty}k^{2r-1+\d}(f_{k-1}(x_{k-1})-f_{k-1}(\ox_{k-1}))<+\infty$ and $\sum_{k=1}^{+\infty}k^{2r-q}\|x_k-x_{k-1}\|^2<+\infty.$

But  according to \eqref{fontos5} we have
$f(x_{k-1})-f(x^*)\le f_{k-1}(x_{k-1})-f_{k-1}(\ox_{k-1})+\frac{c}{2k^p}\|x^*\|^2.$

Further $\sum_{k=1}^{+\infty} k^{2r-1+\d}\frac{c}{2 k^p}\|x^*\|^2<+\infty,$ therefore
$\sum_{k=1}^{+\infty}k^{2r-1+\d}(f(x_{k-1})-\min_{\cH} f)<+\infty.$

Assume now that $r=\frac{q+1}{2}.$ Then \eqref{forEkstrdecayfi} becomes:
\begin{align}\label{forEkstrdecayfipq1str}
&E_{n+1}  +\sum_{k=K_1}^n \xi_k\|u_k\|^2+\sum_{k=K_1}^n m_k(f_{k-1}(x_{k-1})-f_{k-1}(\ox_{k-1}))+\sum_{k=K_1}^n n_k\|x_{k-1}-\ox_{k-1}\|^2\\
\nonumber& +\sum_{k=K_1}^n \eta_k\|x_k-x_{k-1}\|^2+\sum_{k=K_1}^nt_k\|x_{k-1}\|^2\le C\sum_{k=K_1}^n \frac{1}{k},\mbox{ for some }C>0.
\end{align}
But $\sum_{k=1}^n \frac{1}{k}=\mathcal{O}(\ln n)$, hence by dividing \eqref{forEkstrdecayfipq1str} with $\ln n$ we get  at once that there exists $L>0$ such that $\frac{E_{n+1}}{\ln n}<L$. Consequently by arguing analogously as in the case $p<q+1$ we have
$$f_n(x_n)-f_n(\ox_n)=\mathcal{O}(n^{-p-\d}\ln n)$$
and
$$f(x_n)-f(x^*)=\mathcal{O}(n^{-p-\d}\ln n).$$
Further, in this case $\nu_n=\mathcal{O}(1)$ and $\sigma_n=\mathcal{O}(1)$ hence $\frac{1}{\ln n}\|x_n-\ox_n\|^2<L$ and $\frac{1}{\ln n}\|x_n\|^2<L$. Combining the latter relations with the fact that $\frac{v_{n+1}}{\ln n}<L$ we obtain that
$$\|x_n-x_{n-1}\|^2=\mathcal{O}(n^{-p}\ln n).$$
Now, using the definition of $u_n$ we have
$$\|u_n\|^2=\mathcal{O}(n^{-p-2\d}\ln n).$$
Finally, also here the  following average convergence results hold.

$$\limsup_{n\to+\infty}\frac{\sum_{k=1}^n k^{q+\d}(f_{k-1}(x_{k-1})-f_{k-1}(\ox_{k-1}))}{\ln n}<+\infty,$$
$$\limsup_{n\to+\infty}\frac{\sum_{k=1}^n k\|x_{k}-x_{k-1}\|^2}{\ln n}<+\infty$$
and
$$\limsup_{n\to+\infty}\frac{\sum_{k=1}^n k^{q+1+2\d}\|u_k\|^2}{\ln n}<+\infty.$$

Also here, for $\d<0$ it holds $\sum_{k=1}^{+\infty} k^{q+\d}\frac{c}{2 k^p}\|x^*\|^2<+\infty$, hence according to \eqref{fontos5} one has
$$\limsup_{n\to+\infty}\frac{\sum_{k=1}^n k^{q+\d}(f(x_{k-1})-\min_{\cH} f)}{\ln n}<+\infty.$$
\end{proof}

\subsection{Strong convergence results}

Now, in order to show the strong convergence of the sequences generated by \eqref{algo} to an element of minimum norm of the nonempty, convex and closed set $\argmin f$, we state the following results.

\begin{Theorem}\label{strconvergencealgorithm}
  Assume that $0<q< 1$, $1<p< q+1$ and $\l_k=\l k^\d$ with $p-q-1<\d<0,\,\l>0$ or $\d=0$ and $\l\in]0,1[$. Let $(x_k)$ be a sequence generated by \eqref{algo}. Let $x^*$ be the minimal norm element from $\argmin f$. Then, $\liminf_{k\to+\infty}\|x_k-x^*\|=0$. Further, $(x_k)$ converges strongly to $x^*$ whenever $(x_k)$ is in the interior or the complement of the ball $B(0,\|x^*\|)$ for $k$ big enough.
\end{Theorem}
\begin{proof}
We will use the notations and the energy functional $E_k$ used in the proof of Theorem \ref{strconvergencerates}.

{\bf Case  I.} Assume that $\|x_k\|\ge \|x^*\|$ for all $k\ge K_2$, where $K_2\ge K_1$ and $K_1$ was defined in the proof of Theorem \ref{strconvergencerates}.
Let us ad $-\s_k\|x^*\|^2+\s_{k-1}\|x^*\|^2$ to the both side of  \eqref{forEkstrdecay111}. Note that $E_k-\s_{k-1}\|x^*\|^2\ge0$ for all $k>K_2.$ Further, since $\|\ox_k\|\le \|x^*\|$, we get that $\|x_k\|^2-\|\ox_k\|^2\ge 0$ for all $k\ge K_2.$
Then we obtain for all $k>K_2$ that
\begin{align}\label{strconv}
&(E_{k+1}-\s_k\|x^*\|^2)-(E_k-\s_{k-1}\|x^*\|^2) +\xi_k\|u_k\|^2+m_k(f_{k-1}(x_{k-1})-f_{k-1}(\ox_{k-1}))\\
\nonumber& +n_k\|x_{k-1}-\ox_{k-1}\|^2+\eta_k\|x_k-x_{k-1}\|^2+t_k\|x_{k-1}\|^2\\
\nonumber&\le\left(b_{k-1}^2\l_{k-1}(c_{k-1}-c_k)+a_kb_kc_k\right)\|\ox_{k}\|^2-a_{k-1}b_{k-1}\l_{k-1}c_{k-1}\|\ox_{k-1}\|^2+(-\s_k+\s_{k-1})\|x^*\|^2+S_k.
\end{align}
The right hand side of \eqref{strconv} can be written as
\begin{align*}
&\left(b_{k-1}^2\l_{k-1}(c_{k-1}-c_k)+a_kb_kc_k\right)\|\ox_{k}\|^2-a_{k-1}b_{k-1}\l_{k-1}c_{k-1}\|\ox_{k-1}\|^2+(-\s_k+\s_{k-1})\|x^*\|^2+S_k\\
&=\left(\left(b_{k-1}^2\l_{k-1}(c_{k-1}-c_k)+a_kb_kc_k\right)\|\ox_{k}\|^2-\left(b_{k-2}^2\l_{k-2}(c_{k-2}-c_{k-1})+a_{k-1}b_{k-1}c_{k-1}\right)\right)\|\ox_{k-1}\|^2\\
&+\left(b_{k-2}^2\l_{k-2}(c_{k-2}-c_{k-1})+(1-\l_{k-1})a_{k-1}b_{k-1}c_{k-1}\right)\|\ox_{k-1}\|^2+(-\s_k+\s_{k-1})\|x^*\|^2+S_k,
\end{align*}
hence \eqref{strconv} becomes
\begin{align}\label{strconv1}
&(E_{k+1}-\s_k\|x^*\|^2)-(E_k-\s_{k-1}\|x^*\|^2) +\xi_k\|u_k\|^2+m_k(f_{k-1}(x_{k-1})-f_{k-1}(\ox_{k-1}))\\
\nonumber& +n_k\|x_{k-1}-\ox_{k-1}\|^2+\eta_k\|x_k-x_{k-1}\|^2+t_k(\|x_{k-1}\|^2-\|x^*\|^2)\\
\nonumber&\le \left(\left(b_{k-1}^2\l_{k-1}(c_{k-1}-c_k)+a_kb_kc_k\right)\|\ox_{k}\|^2-\left(b_{k-2}^2\l_{k-2}(c_{k-2}-c_{k-1})+a_{k-1}b_{k-1}c_{k-1}\right)\right)\|\ox_{k-1}\|^2\\
\nonumber &+\left(b_{k-2}^2\l_{k-2}(c_{k-2}-c_{k-1})+(1-\l_{k-1})a_{k-1}b_{k-1}c_{k-1}-\s_k+\s_{k-1}-t_k\right)\|x^*\|^2+S_k.
\end{align}

Now, according to \eqref{ordt}, \eqref{ordsigma} and the form of $a_k,\,b_k,\, c_k$ and $\l_k$ we deduce that there exists $K_3>K_2$ such that
\begin{align}\label{forstrconv}
& b_{k-2}^2\l_{k-2}(c_{k-2}-c_{k-1})+(1-\l_{k-1})a_{k-1}b_{k-1}c_{k-1}-\s_k+\s_{k-1}-t_k\\
\nonumber&=\l c(k-2)^{2r+\d}((k-2)^{-p}-(k-1)^{-p})+(1-\l(k-1)^\d)ac(k-1)^{2r-1-p}\\
\nonumber& -ck^{-p}(k^{2r}-\l(k-1)^{2r+\d})-ack^{2r-1-p} +\a ck^{2r-q-p}+c^2k^{2r-2p}\\
\nonumber&+c(k-1)^{-p}((k-1)^{2r}-\l(k-2)^{2r+\d})+ac(k-1)^{2r-1-p} -\a c(k-1)^{2r-q-p}-c^2(k-1)^{2r-2p}\\
\nonumber&-(ac(k-1)^{2r-1-p}-a\l c(k-1)^{2r-1+\d-p})-c((k-1)^{2r-p}-k^{2r-p})\\
\nonumber&-\l c(k-1)^{-p}((k-1)^{2r+\d}-(k-2)^{2r+\d})-\a c(k^{2r-q-p}-(k-1)^{2r-q-p})+c^2(k-1)^{2r-2p}\\
\nonumber&=\l c((k-2)^{2r+\d-p}-(k-1)^{2r+\d-p}+k^{-p}(k-1)^{2r+\d}-(k-1)^{-p}(k-2)^{2r+\d})\\
\nonumber&+ac((k-1)^{2r-1-p}-k^{2r-1-p}) +c^2k^{2r-2p}= c^2k^{2r-2p}+\mathcal{O}(k^{2r-2-p})<C k^{2r-2p}\mbox{ for some }C>0.
\end{align}
Hence, $\sum_{k=K_3}^{+\infty}\left(b_{k-2}^2\l_{k-2}(c_{k-2}-c_{k-1})+(1-\l_{k-1})a_{k-1}b_{k-1}c_{k-1}-\s_k+\s_{k-1}-t_k\right)\|x^*\|^2<+\infty,$
provided $2r-2p<-1.$
So in what follows we assume that $\max(p-\d,1)<2r<\min(q+1,2p-1).$ Then, by summing \eqref{strconv1} by $k=K_3$ to $k=n>K_3$ we obtain that there exists $L>0$ such that
$$\mu_n(f_n(x_n)-f_n(\ox_n))\le L,\mbox{ for all }n> K_3.$$
Now, by \eqref{fontos3} we get
$$\|x_n-\ox_n\|^2<L\frac{2n^p}{c\mu_n}=\frac{L}{\l}\frac{n^p}{(n-1)^{2r+\d}} \mbox{ for all }n> K_3.$$
Consequently, $\|x_n-\ox_n\|\to 0\mbox{ as }n\to+\infty$ which combined with the fact that $\ox_n \to x^*\mbox{ as }n\to+\infty$ lead to
$$\|x_n-x^*\|\to 0\mbox{ as }n\to+\infty.$$

{\bf Case II.}

Assume that there exists $k_0\in \N$ such that $\|x_n\|<\|x^*\|$ for all $n\ge k_0.$

 Now, we take $\bar{x} \in \mathcal{H}$ a weak sequential cluster point of  $(x_n),$ which exists since   $(x_n)$ is bounded. This means that there exists a sequence $\left(k_{n}\right)_{n \in \mathbb{N}} \subseteq[k_0,+\infty)\cap \N$ such that $k_{n} \to +\infty$ and $x_{k_{n}}$ converges weakly to $\bar{x}$ as $n \to +\infty$. According to Theorem \ref{strconvergencerates} and the fact that $f$ is lower semicontinuous one has
$$
f(\bar{x}) \le\liminf_{n \rightarrow+\infty} f\left(x_{k_{n}}\right) = \lim_{n \rightarrow+\infty} f\left(x_{k_{n}}\right)=\min_{\cH}f \, ,$$ hence $\bar{x} \in \operatorname{argmin} f.$
Now, since the norm is weakly lower semicontinuous one has that
$$
\begin{array}{c}
\|\bar{x}\| \leq \liminf _{n \rightarrow+\infty}\left\|x_{k_{n}}\right\| \leq\left\|x^\ast \right\|
\end{array}
$$
which, from the definition of $x^\ast$, implies that $\bar{x}=x^{*}.$ This shows that $(x_n)$ converges weakly to $x^\ast$. So
$$
\left\|x^\ast \right\| \leq \liminf _{n \rightarrow+\infty}\|x_n\| \leq \limsup _{n \rightarrow+\infty}\|x_n\| \leq\left\|x^\ast \right\|,$$
hence we have
$$ \lim _{n \rightarrow+\infty}\|x_n\|=\left\|x^\ast \right\|.$$
From  the previous relation and the fact that $x_n\rightharpoonup x^\ast$ as $n \to +\infty,$ we obtain the strong convergence, that is
$$
\lim _{n \rightarrow+\infty} x_n=x^\ast.$$

{\bf Case  III.} We suppose that there exists $k_0\in\N$ such that for every $n \geq k_0$ there exists $l \geq n$ such that $\left\|x^\ast \right\|>\|x_l\|$ and also there exists $m \geq n$ such that $\left\|x^{*}\right\| \leq\|x_m\|$.

So let $k_1\ge k_0$ and $l_1\ge k_1$ such that $\left\|x^\ast \right\|>\|x_{l_1}\|.$
Let $k_2>l_1$ and $l_2\ge k_2$ such that $\left\|x^\ast \right\|>\|x_{l_2}\|.$ Continuing the procedure we obtain $(x_{l_n})$, a subsequence of $(x_n)$ with the property that $\|x_{l_n}\|<\|x^*\|$ for all $n\in\N.$ Now reasoning as in {\bf Case II} we obtain that
$\lim _{n \rightarrow+\infty} x_{l_n}=x^\ast.$
Consequently,
$$\liminf_{k \rightarrow+\infty} \|x_n-x^\ast\|=0.$$
\end{proof}

\subsection{Full strong convergence for the case $\d=0,\,\l=1$}

Now we are able  to show that in case $\l=1$ the sequences generated by Algorithm \ref{algo} converges strongly to the minimum norm minimizer of the objective function $f.$ The following result is our main result of the present section.

\begin{Theorem}\label{strconvergencelambda1}
Assume that $0<q< 1$, $1<p< q+1$, $\l_k\equiv 1$. Let $(x_k)$ be a sequence generated by \eqref{algo}. For every $k\ge 2$ let us denote by $u_k$ the element  from $\p f(x_k)$ that satisfies \eqref{formx} with equality, i.e.,
$$x_{k}=\a_{k-1}(x_{k-1}-x_{k-2})- u_k +\left(1- c_{k-1}\right) x_{k-1}.$$  Then   the following  results are valid.
\begin{enumerate}
\item[(i)]  If $p\le 2q$ then $\|x_n-\ox_n\|=\mathcal{O}(n^\frac{q-1}{2})$ as $n\to+\infty,$ hence $\lim_{n\to+\infty}x_n=x^*.$
Further,
$\|x_n-x_{n-1}\|^2,\,\|u_n\|^2\in \mathcal{O}(n^{q-p-1})\mbox{ as }n\to+\infty$ and $f(x_{n})-\min_{\cH}f= \mathcal{O}(n^{-p})\mbox{ as }n\to+\infty.$

\item[(ii)] If $2q< p\le\frac{3q+1}{2}$ then  $\|x_n-\ox_n\|=\mathcal{O}(n^\frac{q-1}{2})$ as $n\to+\infty$ and $\lim_{n\to+\infty}x_n=x^*.$ Further,
$f_{n}(x_{n})-f_{n}(\ox_{n}),\,\|x_n-x_{n-1}\|^2,\,\|u_n\|^2\in \mathcal{O}(n^{-q-1})\mbox{ as }n\to+\infty$
and $f(x_{n})-\min_{\cH}f= \mathcal{O}(n^{-p})\mbox{ as }n\to+\infty.$ The following sum estimates also hold.
$\sum_{k=1}^{+\infty} k^q (f_{k}(x_{k})-f_{k}(\ox_{k}))<+\infty,$ $\sum_{k=1}^{+\infty} k^{2q}\|u_k\|^2<+\infty$ and
$\sum_{k=1}^{+\infty} k^q\|x_{k+1}-x_{k}\|^2<+\infty.$
\item[(iii)] If $\frac{3q+1}{2}<p<q+1$, then
 $\|x_n-\ox_n\|=\mathcal{O}(n^{p-q-1})$ as $n\to+\infty$, hence
$\lim_{n\to+\infty}x_n=x^*.$ Further,
$f_{n}(x_{n})-f_{n}(\ox_{n}),\,\|x_n-x_{n-1}\|^2,\,\|u_n\|^2\in \mathcal{O}(n^{2p-4q-2})\mbox{ as }n\to+\infty.$
Additionally, if  $\frac{3q+1}{2}<p<\frac{4q+2}{3}$, then $f(x_{n})-\min_{\cH}f= \mathcal{O}(n^{-p})\mbox{ as }n\to+\infty$
and if $\frac{4q+2}{3}\le p<q+1$, then $f(x_{n})-\min_{\cH}f= \mathcal{O}(n^{2p-4q-2})\mbox{ as }n\to+\infty.$
Moreover,
$\sum_{k=1}^{+\infty} k^q (f_{k}(x_{k})-f_{k}(\ox_{k}))<+\infty,$ $\sum_{k=1}^{+\infty} k^{2q}\|u_k\|^2<+\infty$ and
$\sum_{k=1}^{+\infty} k^q\|x_{k+1}-x_{k}\|^2<+\infty.$
\end{enumerate}
\end{Theorem}
\begin{proof} We use the notations from the proof of Theorem \ref{strconvergencerates}. Then, for $\l=1,\,\d=0$ \eqref{forEkstrdecay11} becomes
\begin{align}\label{fulstr1}
&v_{k+1}-v_k   +2b_{k-1}^2(f_k(x_k)-f_k(\ox_k))-2b_{k-2}^2(f_{k-1}(x_{k-1})-f_{k-1}(\ox_{k-1}))\\
\nonumber&-(2b_{k-1}^2-2a_{k-1}b_{k-1}-2b_{k-2}^2)(f_{k-1}(x_{k-1})-f_{k-1}(\ox_{k-1}))+b_{k-1}^2\|u_k\|^2\\
\nonumber& +(-a_k^2-\a_k a_kb_k+a_kb_kc_k+a_{k-1}b_{k-1})\|x_k-\ox_{k}\|^2\\
\nonumber& -(-a_{k-1}^2-\a_{k-1} a_{k-1}b_{k-1}+a_{k-1}b_{k-1}c_{k-1}+a_{k-2}b_{k-2})\|x_{k-1}-\ox_{k-1}\|^2\\
\nonumber&+(-\a_{k-1} a_{k-1}b_{k-1}+a_{k-1}b_{k-1}c_{k-1}+a_{k-2}b_{k-2}-(1+s_{k-1})(a_{k-1}b_{k-1}-\a_k a_kb_k))\|x_{k-1}-\ox_{k-1}\|^2\\
\nonumber& -(\a_k^2b_k^2+\a_k a_kb_k-\a_kb_k^2c_k-b_{k-1}^2+3a_{k-1}b_{k-1})\|x_k-x_{k-1}\|^2\\
\nonumber&\le (b_k^2c_k^2-\a_kb_k^2c_k -a_kb_kc_k+ b_{k-1}^2 c_k)\|x_k\|^2\\
\nonumber& -(b_{k-1}^2c_{k-1}^2-\a_{k-1}b_{k-1}^2c_{k-1} -a_{k-1}b_{k-1}c_{k-1}+ b_{k-2}^2 c_{k-1})\|x_{k-1}\|^2\\
\nonumber&+(b_{k-1}^2c_{k-1}^2+\a_kb_k^2c_k-\a_{k-1}b_{k-1}^2c_{k-1} + b_{k-2}^2 c_{k-1}-b_{k-1}^2c_{k-1})\|x_{k-1}\|^2\\
\nonumber&+\left(b_{k-1}^2(c_{k-1}-c_k)+a_kb_kc_k\right)\|\ox_{k}\|^2- a_{k-1}b_{k-1}c_{k-1}\|\ox_{k-1}\|^2\\
\nonumber&+\left(\left(1+\frac{1}{s_{k-1}}\right)(a_{k-1}b_{k-1}-\a_k a_kb_k)+\frac{a_{k-1}b_{k-1}}{2}- b_{k-1}^2c_{k-1}\right)\|\ox_{k}-\ox_{k-1}\|^2.
\end{align}

We will assume from now on that $a_k\equiv a,\,\a>a>0$ and $b_k=k^q.$ Then, concerning the right hand side of \eqref{fulstr1} we conclude the following.
$$-\sigma_k=b_k^2c_k^2-\a_kb_k^2c_k -a_kb_kc_k+ b_{k-1}^2 c_k\ge 0\mbox{ after an index }k\mbox{ big enough.}$$
Further, $-\sigma_k=\mathcal{O}(k^{q-p}).$
Note that for $k$ big enough one has
$$-t_k=b_{k-1}^2c_{k-1}^2+\a_kb_k^2c_k-\a_{k-1}b_{k-1}^2c_{k-1} + b_{k-2}^2 c_{k-1}-b_{k-1}^2c_{k-1}\le 0.$$

Now, since $\|\ox_k\|\le\|x^*\|$ we conclude that there exists $C_1>0$ such that
$$b_{k-1}^2(c_{k-1}-c_k)\|\ox_{k}\|^2\le C_1 k^{2q-p-1}\mbox{ for }k\mbox{ big enough.}$$
We recall that $S_k=\left(\left(1+\frac{(k-1)^{p-q}}{s}\right)(a_{k-1}b_{k-1}-\a_k a_kb_k)+\frac{a_{k-1}b_{k-1}}{2}\right)\|\ox_{k}-\ox_{k-1}\|^2$ and by using \eqref{lfos} we conclude that there exists $C_2>0$ such that for $k$ big enough one has
$$S_k\le C_2 k^{\max(p-q-2,q-2)}.$$

Consider now the energy functional $e_k=\mu_{k-1}(f_{k-1}(x_{k-1})-f_{k-1}(\ox_{k-1}))+ v_k+\nu_{k-1}\|x_{k-1}-\ox_{k-1}\|^2.$ Obviously for our setting, one has $\mu_k=\mathcal{O}(k^{2q})$ and $\nu_k=\mathcal{O}(1).$
Then, \eqref{fulstr1} yields
\begin{align}\label{fulstr2}
&e_{k+1}-e_k +m_k (f_{k-1}(x_{k-1})-f_{k-1}(\ox_{k-1}))+\xi_k\|u_k\|^2
+n_k\|x_{k-1}-\ox_{k-1}\|^2 +\eta_k\|x_k-x_{k-1}\|^2\\
\nonumber&\le -\sigma_k\|x_k\|^2
+\sigma_{k-1}\|x_{k-1}\|^2+a_kb_kc_k\|\ox_{k}\|^2- a_{k-1}b_{k-1}c_{k-1}\|\ox_{k-1}\|^2\\
\nonumber&\,\,+C_1 k^{2q-p-1}+C_2 k^{\max(p-q-2,q-2)}.
\end{align}

Note that for $k$ big enough one has
$m_k= -(2b_{k-1}^2-2a_{k-1}b_{k-1}-2b_{k-2}^2)\ge 0$ and $m_k=\mathcal{O}(k^q)$ as $k\to+\infty,$
$n_k=-\a_{k-1} a_{k-1}b_{k-1}+a_{k-1}b_{k-1}c_{k-1}+a_{k-2}b_{k-2}-(1+s_{k-1})(a_{k-1}b_{k-1}-\a_k a_kb_k)\ge 0$ and $n_k=\mathcal{O}(k^{q-p})$ as $k\to+\infty,$
$\xi_k=b_{k-1}^2=(k-1)^{2q}\ge 0$ and $\xi_k=\mathcal{O}(k^{2q})$ as $k\to+\infty,$
further
$\eta_k=-\a_k^2b_k^2-\a_k a_kb_k+\a_kb_k^2c_k+b_{k-1}^2-3a_{k-1}b_{k-1}\ge 0$ and $\eta_k=\mathcal{O}(k^{q})$ as $k\to+\infty$.
By using the fact that
\begin{align*}
v_k&=\| a_{k-1}(x_{k-1}-\ox_{k-1}) +b_{k-1} (x_k-x_{k-1}+ u_k)\|^2\\
&\le2a^2\|x_{k-1}-\ox_{k-1}\|^2+4(k-1)^{2q}\|x_k-x_{k-1}\|^2+4(k-1)^{2q}\|u_k\|^2,
\end{align*}
we deduce that there exists $H>0$ such that
$$\frac{H}{k^{\b}}e_k\le m_k (f_{k-1}(x_{k-1})-f_{k-1}(\ox_{k-1}))+\xi_k\|u_k\|^2
+n_k\|x_{k-1}-\ox_{k-1}\|^2 +\eta_k\|x_k-x_{k-1}\|^2,$$
where $\b=\max(q,p-q)<1.$

Consequently, according to \eqref{fulstr2} there exists an index $K_0\in\N$ such that for all $k> K_0$ it holds
\begin{align}\label{fulstr3}
&e_{k+1}-e_k +\frac{H}{k^{\b}}e_k\le -\sigma_k\|x_k\|^2
+\sigma_{k-1}\|x_{k-1}\|^2+a_kb_kc_k\|\ox_{k}\|^2- a_{k-1}b_{k-1}c_{k-1}\|\ox_{k-1}\|^2\\
\nonumber&\,\,+C_1 k^{2q-p-1}+C_2 k^{\max(p-q-2,q-2)}.
\end{align}
Now, by multiplying \eqref{fulstr3} with $\pi_k=\frac{1}{\prod_{i=K_0}^k\left(1-\frac{H}{i^{\b}}\right)}$ we obtain
\begin{align}\label{fulstr4}
\pi_ke_{k+1}-\pi_{k-1}e_k \le& \pi_k((-\sigma_k)\|x_k\|^2-(-\sigma_{k-1})\|x_{k-1}\|^2)\\
\nonumber&+\pi_k(a_kb_kc_k\|\ox_{k}\|^2- a_{k-1}b_{k-1}c_{k-1}\|\ox_{k-1}\|^2)\\
\nonumber&\,\,+C_1\pi_k k^{2q-p-1}+C_2 \pi_k k^{\max(p-q-2,q-2)}.
\end{align}
Now, by summing \eqref{fulstr4} from $k=K_0+1$ to $n>K_0+1$ big enough and using Lemma \ref{pkstuff} we obtain that there exist some positive constants still denoted  by $C_1,C_2,C_3$ such that
$$\pi_{n}e_{n+1}\le \pi_n(-\sigma_n)\|x_n\|^2+\pi_n a_nb_nc_n\|\ox_{n}\|^2+C_1\pi_n n^{2q-p-1+\b}+C_2 \pi_n n^{\max(p-q-2+\b,q-2+\b)}+C_3.$$
Now, taking into account that $(-\sigma_n),\,(a_nb_nc_n)=\mathcal{O}(n^{q-p})$ as $n\to+\infty$ and according to Theorem \ref{boundedness} $(x_n)$ is bounded and $\|\ox_n\|\le \|x^*\|$, the above relation leads to
\begin{equation}\label{foren}
e_{n+1}\le C_0 n^{q-p}+C_1 n^{2q-p-1+\b}+C_2 n^{\max(p-q-2+\b,q-2+\b)}+\frac{C_3}{\pi_n}<C(n^{q-p}+n^{2q-p-1+\b}+ n^{\max(p-q-2+\b,q-2+\b)}),
\end{equation}
for some constant $C>0.$

Let us discuss the order of the right hand side of \eqref{foren}.

If $\max(p-q-2+\b,q-2+\b)=q-2+\b$, that is, $p\le 2q$ then $\b=\max(p-q,q)=q$, hence $\max(p-q-2+\b,q-2+\b)=2q-2.$ Obviously by assumption $2q-p-1+\b=3q-p-1>2q-2$, further, since $q\ge \frac{p}{2}>\frac12$ one has $3q-p-1>q-p$ so the right hand side of \eqref{foren} is less than $Cn^{3q-p-1}$ for a constant $C>0$ appropriately chosen.

If $\max(p-q-2+\b,q-2+\b)=p-q-2+\b$, that is, $p\ge 2q$ then $\b=\max(p-q,q)=p-q$, hence $\max(p-q-2+\b,q-2+\b)=2p-2q-2.$
Obviously, the $2q-p-1+\b=q-1>q-p,$ hence the right hand side of \eqref{foren} is less than $Cn^{q-1}$  provided $2q\le p\le\frac{3q+1}{2}$ and the right hand side of \eqref{foren} is less than  $Cn^{2p-2q-2}$ provided $\frac{3q+1}{2}<p<q+1$ for a constant $C>0$ appropriately chosen.

So using \eqref{fontos3}, \eqref{foren} and the form of $e_{n+1}$ we conclude the following.

{\bf a.} If $p\le 2q$ then for some $C'>0$ it holds
$$\|x_n-\ox_n\|^2\le \frac{2n^p}{c}(f_{n}(x_{n})-f_{n}(\ox_{n}))\le \frac{2n^p}{c\mu_n}e_{n+1}\le C' n^{q-1}.$$
Consequently, $\|x_n-\ox_n\|=\mathcal{O}(n^\frac{q-1}{2})$ as $n\to+\infty.$ Since $\ox_n\to x^*$ as $n\to+\infty$, we obtain in particular that $\lim_{n\to+\infty}x_n=x^*.$

Further, $f_{n}(x_{n})-f_{n}(\ox_{n})\le \frac{1}{c\mu_n}e_{n+1}$ and $v_{n+1}\le e_{n+1}$, hence
$$f_{n}(x_{n})-f_{n}(\ox_{n}),\,\|x_n-x_{n-1}\|^2,\,\|u_n\|^2\in \mathcal{O}(n^{q-p-1})\mbox{ as }n\to+\infty.$$

According to \eqref{fontos5} we have $f(x_n)-\min_{\cH}f\le f_n(x_n)-f_n(\ox_n)+\frac{c}{2n^p}\|x^*\|^2$, and since $q-p-1<-p$, we obtain that
$f(x_{n})-\min_{\cH}f= \mathcal{O}(n^{-p})\mbox{ as }n\to+\infty.$

{\bf b.} If $2q< p\le\frac{3q+1}{2}$ then by using the fact that $\nu_n=\mathcal{O}(1)$ we obtain from \eqref{foren} that
$$\|x_n-\ox_n\|^2\le  \frac{1}{\nu_n}e_{n+1}\le C' n^{q-1}, \mbox{ for some }C'>0.$$
Consequently, $\|x_n-\ox_n\|=\mathcal{O}(n^\frac{q-1}{2})$ as $n\to+\infty$ and since $q<1$ we obtain in particular that $\lim_{n\to+\infty}x_n=x^*.$ Analogously to the previous case, one can deduce that
$$f_{n}(x_{n})-f_{n}(\ox_{n}),\,\|x_n-x_{n-1}\|^2,\,\|u_n\|^2\in \mathcal{O}(n^{-q-1})\mbox{ as }n\to+\infty$$
and $f(x_{n})-\min_{\cH}f= \mathcal{O}(n^{-p})\mbox{ as }n\to+\infty.$

{\bf c.} If $\frac{3q+1}{2}<p<q+1$, then by the same argument as in the previous case we deduce that
 $\|x_n-\ox_n\|=\mathcal{O}(n^{p-q-1})$ as $n\to+\infty$, hence
$\lim_{n\to+\infty}x_n=x^*.$ Further, one has
$$f_{n}(x_{n})-f_{n}(\ox_{n}),\,\|x_n-x_{n-1}\|^2,\,\|u_n\|^2\in \mathcal{O}(n^{2p-4q-2})\mbox{ as }n\to+\infty.$$
Here, by using \eqref{fontos5}, concerning the rate of the potential energy $f(x_{n})-\min_{\cH}f$ we conclude the following.

In one hand, if $-p>2p-4q-2$, that is $2q<p<\frac{4q+2}{3}$, then $f(x_{n})-\min_{\cH}f= \mathcal{O}(n^{-p})\mbox{ as }n\to+\infty.$

On the other hand, if $\frac{4q+2}{3}\le p<q+1$, then $f(x_{n})-\min_{\cH}f= \mathcal{O}(n^{2p-4q-2})\mbox{ as }n\to+\infty.$
\vskip0.3cm

In order to obtain sum estimates, let us return to \eqref{fulstr2} which holds from an index $K_0$ big enough. By summing \eqref{fulstr2} from $k=K_0$ to $k=n$ we obtain
\begin{align}\label{fulstr22}
&e_{n+1} +\sum_{k=K_0}^n m_k (f_{k-1}(x_{k-1})-f_{k-1}(\ox_{k-1}))+\sum_{k=K_0}^n \xi_k\|u_k\|^2
+\sum_{k=K_0}^n n_k\|x_{k-1}-\ox_{k-1}\|^2\\
\nonumber& +\sum_{k=K_0}^n \eta_k\|x_k-x_{k-1}\|^2\le -\sigma_n\|x_n\|^2+a_nb_nc_n\|\ox_{n}\|^2+C_1 \sum_{k=K_0}^n k^{2q-p-1}\\
\nonumber& +C_2\sum_{k=K_0}^n  k^{\max(p-q-2,q-2)}+C_3,\mbox{ for some }C_3>0.
\end{align}
Now, since $(x_n),\,(\ox_n)$ are bounded and $\sigma_n,\,a_nb_nc_n\in\mathcal{O}(n^{q-p})$ as $n\to+\infty,$ further $q<1<p<1+q$, we deduce that for $p>2q$ the right hand side of \eqref{fulstr22} is finite. So taking into account the form of $m_k,\eta_k$ and $\xi_k$ we obtain that
$\sum_{k=1}^{+\infty} k^q (f_{k}(x_{k})-f_{k}(\ox_{k}))<+\infty,$ $\sum_{k=1}^{+\infty} k^{2q}\|u_k\|^2<+\infty$ and
$\sum_{k=1}^{+\infty} k^q\|x_{k+1}-x_{k}\|^2<+\infty.$
\end{proof}

\section{Conclusions, perspectives}

In the present paper we showed that the constellation  $q=1,\,\l_k\equiv 1$ is not necessarily the best choice for Algorithm \eqref{algo} since in case $0<q<1$ the control on the stepsize parameter $\l_k$ allows us to obtain arbitrary  rate for the potential energy $f(x_k)-\min_{\cH}f$. Further, our analysis reveals that the inertial parameter $\a_k$, the stepsize $\l_k$ and the Tikhonov regularization parameter $c_k$ are strongly correlated: in case $q+1<p,\,\d\ge 0$ weak convergence of the generated sequences and fast convergence of  the function values can be obtained, meanwhile in case $p<q+1,\,\d\le 0$ strong convergence results for the generated sequences and fast convergence of  the function values can be provided.

Another important achievement of the present paper is that for the case $\l_k\equiv 1,\,p<q+1$ we succeeded to obtain "full" strong convergence of the generated sequences to the minimal norm solution $x^*$, that is $\lim_{k\to+\infty}\|x_k-x^*\|=0.$ For the same constellation of parameters, we also obtained fast convergence of  the function values and velocity and some  sum estimates. Due to our best knowledge this is the first result of this type in the literature concerning discrete dynamical systems, however in continuous case some similar results have already been obtained in the recent papers \cite{ABCR,L-jde,BCLstr}. Nevertheless, in order to obtain strong convergence we had to develop some original new techniques.

 In our context, one can observe that the case $p=q+1$ is  critical in the sense that separates the two cases: the case when we obtain fast convergence of the function values and weak convergence of the generated sequences to a minimizer and the case when the strong convergence of the generated sequences to a minimizer of minimum norm is assured. However, even in this case we can obtain  fast convergence of  the function values and velocity and also sum estimates, both for the case $\d\ge 0$ and $\d<0.$ These facts are in concordance with the results obtained for continuous dynamics in \cite{ACR}, \cite{BCL} and \cite{AL-siopt}.

  Some other subjects for future investigations are the gradient type algorithms obtained via explicit discretization from  \eqref{DynSys} and the dynamical systems studied in the papers mentioned above.

\appendix
\section{Auxiliary results}\label{app}

The following lemma summarizes several important results which are behind the Tikhonov regularization techniques and are used in our proofs.

\begin{lemma}\label{lfosapen} Let $f:\mathcal{H}\to\overline{\mathbb{R}}$ a proper, convex and lsc function and let $(\e_k)$ a positive non-increasing sequence that converges to $0.$
By  $\ox_{k}$ we denote the unique solution of the strongly convex minimization problem
\begin{align*}
 \min_{x \in \mathcal{H}} \left( f(x) + \frac{\e_k}{2} \| x \|^2 \right).
\end{align*}

Then, for all $k\ge 1$ one has
$$\frac{\e_k-\e_{k+1}}{\e_{k+1}}\<\ox_k,\ox_{k+1}-\ox_{k}\>\ge\|\ox_{k+1}-\ox_{k}\|^2$$
and
$$\frac{\e_k-\e_{k+1}}{\e_k}\<\ox_{k+1},\ox_{k+1}-\ox_{k}\>\ge\|\ox_{k+1}-\ox_{k}\|^2.$$
Consequently, the sequence $(\|\ox_k\|)_{k\ge 1}$ is non-decreasing and  one has $\<\ox_{k+1},\ox_k\>\ge 0$ for all $k\ge 1.$ Additionally, the following statements hold for all $k\ge 1$.
\begin{itemize}
\item[a)]  $\|\ox_{k+1}\|^2-\|\ox_k\|^2\ge \frac{\e_{k}+\e_{k+1}}{\e_{k}-\e_{k+1}}\|\ox_{k+1}-\ox_k\|^2.$
\item[b)]  $\|\ox_k\|^2+\frac{\e_{k+1}}{\e_k-\e_{k+1}}\|\ox_{k+1}-\ox_k\|^2\le\<\ox_{k+1},\ox_k\>\le \|\ox_{k+1}\|^2-\frac{\e_{k}}{\e_k-\e_{k+1}}\|\ox_{k+1}-\ox_k\|^2.$
    \item[c)] $\left\|\ox_{k+1}-\ox_k\right\|\le\min\left(\frac{\e_k-\e_{k+1}}{\e_{k+1}}\|\ox_{k}\|,\frac{\e_k-\e_{k+1}}{\e_k}\|\ox_{k+1}\|\right).$
\end{itemize}
\end{lemma}
\begin{proof} Since $\ox_{k}$ is the unique minimum of the strongly convex function $f_k(x)=f(x)+\frac{\e_k}{2}\|x\|^2,$ obviously one has
\begin{equation}\label{fontos0apen}
\p f_k(\ox_{k})=\p f(\ox_{k})+\e_k\ox_{k}\ni0.
\end{equation}
Hence, we have
$-\e_k\ox_{k}\in \p f(\ox_{k})$ and $-\e_{k+1}\ox_{k+1}\in \p f(\ox_{k+1})$ and by using the monotonicity of $\p f$ we get
$$\left\<-\e_{k+1}\ox_{k+1}+\e_{k}\ox_{k},\ox_{k+1}-\ox_{k}\right\>\ge 0.$$
In other words
$$-\e_{k+1}\<\ox_{k+1}-\ox_{k},\ox_{k+1}-\ox_{k}\>+\left(\e_{k}-\e_{k+1}\right)\<\ox_k,\ox_{k+1}-\ox_{k}\>\ge 0$$
or, equivalently
\begin{equation}\label{fontos1apen}
\frac{\e_{k}-\e_{k+1}}{\e_{k+1}}\<\ox_k,\ox_{k+1}-\ox_{k}\>\ge\|\ox_{k+1}-\ox_{k}\|^2.
\end{equation}
But, $\<\ox_k,\ox_{k+1}-\ox_{k}\>=-\|\ox_{k+1}-\ox_{k}\|^2+\<\ox_{k+1},\ox_{k+1}-\ox_{k}\>$ hence
$$\frac{\e_{k}-\e_{k+1}}{\e_{k+1}}\<\ox_{k+1},\ox_{k+1}-\ox_{k}\>\ge\frac{\e_{k}}{\e_{k+1}}\|\ox_{k+1}-\ox_{k}\|^2.$$
Equivalently, we can write
\begin{equation}\label{fontos2apen}
\frac{\e_{k}-\e_{k+1}}{\e_{k}}\<\ox_{k+1},\ox_{k+1}-\ox_{k}\>\ge\|\ox_{k+1}-\ox_{k}\|^2.
\end{equation}
In order to prove a) note that
$\<\ox_k,\ox_{k+1}-\ox_{k}\>=\frac12(\|\ox_{k+1}\|^2-\|\ox_{k}\|^2-\|\ox_{k+1}-\ox_{k}\|^2),$
hence \eqref{fontos1apen} leads to
\begin{equation}\label{fontos3apen}
\|\ox_{k+1}\|^2-\|\ox_{k}\|^2\ge \frac{\e_{k}+\e_{k+1}}{\e_{k}-\e_{k+1}}\|\ox_{k+1}-\ox_{k}\|^2.
\end{equation}

Observe that b) is actually equivalent to \eqref{fontos1apen} and \eqref{fontos2apen}.

For proving c) we simply use  in \eqref{fontos1apen} and \eqref{fontos2apen} the Cauchy-Schwarz inequality and simplify with $\|\ox_{k+1}-\ox_{k}\|.$

Finally, note that a) implies that the sequence $(\|\ox_k\|)_{k\ge 1}$ is non-decreasing and b) implies that $\<\ox_{k+1},\ox_k\>\ge 0$ for all $k\ge 1.$
\end{proof}

The following result is used in the proofs of our strong convergence results.
\begin{lemma}\label{pkstuff} Let $H>0,\,0<\b$ and for $K_0\in\N,\, K_0>H^{\frac{1}{\b}}$ consider the sequence $\pi_k=\frac{1}{\prod_{i=K_0}^k \left(1-\frac{H}{i^\b}\right)}.$ Then obviously $(\pi_k)$ is a positive non-decreasing sequence and has the following properties.
\begin{itemize}
\item[a)] If $\b\in]0,1[$ then there exists $C_1,C_2>0$ such that  after an index $n_0\in\N$ it holds $$e^{C_1 n^{1-\b}}\le\pi_n\le e^{C_2 n^{1-\b}},\mbox{ for all }n\ge n_0.$$ Further, if $\b=1$ then $\pi_n=\mathcal{O}(n^H)$ as $n\to+\infty$.
\item[b)] If $\b\in]0,1[$ then for all $\g\in\R$ and $n$ big enough, one has $$C_1n^{\g+\b}\pi_n\le\sum_{k=K_0}^n k^\g\pi_k\le C_2n^{\g+\b}\pi_n,\mbox{ for some }C_1,C_2>0.$$
\item[c)] For every nonegative sequence $(a_k)$ one has
$$\sum_{k=K_0+1}^n \pi_k(a_{k}-a_{k-1})\le a_{n}\pi_{n}.$$
\end{itemize}

\end{lemma}
\begin{proof} In case  $\b\in]0,1[$, by applying the Ces\`aro-Stolz theorem, we have $$\lim_{n\to+\infty}\frac{\ln\pi_n}{n^{1-\b}}=\lim_{n\to+\infty}\frac{\ln\frac{\pi_{n+1}}{\pi_n}}{(n+1)^{1-\b}-n^{1-\b}}=\lim_{n\to+\infty}\frac{\frac{H}{(n+1)^\b}\ln\left(1-\frac{H}{(n+1)^\b}\right)^{-\frac{(n+1)^\b}{H}}}{(n+1)^{1-\b}-n^{1-\b}}$$
But $\lim_{n\to+\infty}\frac{\frac{H}{(n+1)^\b}}{(n+1)^{1-\b}-n^{1-\b}}=\frac{H}{1-\b}$ and   $\lim_{n\to+\infty}\ln\left(1-\frac{H}{(n+1)^\b}\right)^{-\frac{(n+1)^\b}{H}}=1$, hence
$$\lim_{n\to+\infty}\frac{\ln\pi_n}{n^{1-\b}}=\frac{H}{1-\b}.$$
In other words, for every $\e>0$ there exists $n_0\in\N$ such that for all $n\ge n_0$ one has
$$e^{\left(\frac{H}{1-\b}-\e\right)n^{1-\b}}\le \pi_n\le e^{\left(\frac{H}{1-\b}+\e\right)n^{1-\b}}$$
and the conclusion follows.

In case  $\b=1$, by applying the Ces\`aro-Stolz theorem, we have
$$\lim_{n\to+\infty}\frac{\ln\pi_n}{\ln n}=\lim_{n\to+\infty}\frac{\frac{H}{n+1}\ln\left(1-\frac{H}{n+1}\right)^{-\frac{n+1}{H}}}{\frac{1}{n}\ln\left(1+\frac{1}{n}\right)^n}=H$$
and the conclusion follows.

b) Note that it is enough to show that $\lim_{n\to+\infty}\frac{\sum_{k=K_0+1}^n k^\g\pi_k}{n^{\g+\b}\pi_n}$ exists and is finite. Observe that  according to a) one has $\lim_{n\to+\infty} n^{\g+\b}\pi_n=+\infty$ for every $\g\in\R.$ Further, for every $\g\in\R$ one has
$\frac{(n+1)^{\g+\b}\pi_{n+1}}{ n^{\g+\b}\pi_n}=\left(1+\frac{1}{n}\right)^{\g+\b}\frac{(n+1)^\b}{(n+1)^\b-H}>1$, hence the sequence $(n^{\g+\b}\pi_n)$ is increasing. Consequently Ces\`aro-Stolz theorem can be applied in order to find the limit $\lim_{n\to+\infty}\frac{\sum_{k=K_0+1}^n k^\g\pi_k}{n^{\g+\b}\pi_n}.$
We have
$$\lim_{n\to+\infty}\frac{\sum_{k=K_0+1}^n k^\g\pi_k}{n^{\g+\b}\pi_n}=\lim_{n\to+\infty}\frac{(n+1)^{\g}\pi_{n+1}}{(n+1)^{\g+\b}\pi_{n+1}-n^{\g+\b}\pi_{n}}=\lim_{n\to+\infty}\frac{(n+1)^{\g}}{(n+1)^{\g+\b}-n^{\g+\b}\frac{\pi_{n}}{\pi_{n+1}}}.$$
Further, $$\lim_{n\to+\infty}\frac{(n+1)^{\g}}{(n+1)^{\g+\b}-n^{\g+\b}\frac{\pi_{n}}{\pi_{n+1}}}=\lim_{n\to+\infty}\frac{1}{\frac{(n+1)^{\g+\b}-n^{\g+\b}}{(n+1)^{\g}}+\frac{Hn^{\g+\b}}{(n+1)^{\g+\b}}}=\frac{1}{H}.
$$

c) We have $\pi_ka_{k-1}=\pi_{k-1}a_{k-1}+\frac{H}{k^\b-H}\pi_{k-1}a_{k-1}$, hence
$$\sum_{k=K_0+1}^n \pi_k(a_{k}-a_{k-1})\le\sum_{k=K_0+1}^n (\pi_ka_{k}-\pi_{k-1}a_{k-1})\le a_{n}\pi_{n}.$$
\end{proof}

{\bf  Conflicts of interests.} We have no conflicts of interest to disclose.


\begin{thebibliography}{10}

\bibitem{AAD} V. Apidopoulos, J.F. Aujol, C. Dossal, {\it Convergence rate of inertial forward-backward algorithm beyond Nesterov’s rule}, Mathematical Programming 180,  137–156 (2020)
    
\bibitem{AAD1} V. Apidopoulos, J.F. Aujol, C. Dossal, {\it The differential inclusion modeling FISTA algorithm and optimality of convergence rate in the case b<=3}, SIAM Journal on Optimization 28, 551-574 {2018}
    
\bibitem{AL-siopt} {\sc C.D. Alecsa, S.C. L\' aszl\' o}, {\it Tikhonov regularization of a perturbed heavy ball system with vanishing damping}, SIAM J. OPTIM. 31(4),  2921-2954 (2021)

\bibitem{ABCR} {\sc H. Attouch, A. Balhag, Z. Chbani, H. Riahi},  {\it Damped inertial dynamics with vanishing Tikhonov regularization: Strong asymptotic convergence towards the minimum norm solution}, Journal of Differential Equations  311, 29-58 (2022)

\bibitem{ABCRamop} H. Attouch, A. Balhag, Z. Chbani, H. Riahi, {\it Accelerated Gradient Methods Combining Tikhonov Regularization with Geometric Damping Driven by the Hessian}, Appl Math Optim, 88, 29, (2023)
    
\bibitem{ABotCest}  {\sc H. Attouch, R.I. Bo\c t,  E.R. Csetnek}, {\it Fast optimization via  inertial dynamics  with closed-loop damping}, Journal of the European Mathematical Society  (2022), DOI 10.4171/JEMS/1231

\bibitem{abc2} {\sc H. Attouch, L.M. Brice\~no-Arias, P.L. Combettes},  {\it A strongly convergent primal-dual method for nonoverlapping domain decomposition},  Numerische Mathematik 133(3), 443-470 (2016)

\bibitem{ACPR} {\sc H. Attouch, Z. Chbani, J. Peypouquet, P. Redont}, {\it Fast convergence of inertial dynamics and algorithms with asymptotic vanishing viscosity}, Mathematical Programming 168 (1-2), 123-175 (2018)

\bibitem{ACR} {\sc H. Attouch,  Z. Chbani, H. Riahi}, {\it Combining fast inertial dynamics for convex optimization with Tikhonov regularization}, J. Math. Anal. Appl 457, 1065-1094 (2018)

\bibitem{ACR1} {\sc H. Attouch,  Z. Chbani, H. Riahi}, {\it Fast proximal methods via time scaling of damped inertial dynamics}, SIAM Journal on Optimization 29(3), 2227-2256 (2019)

\bibitem{ACR2} {\sc H. Attouch,  Z. Chbani, H. Riahi}, {\it Accelerated gradient methods with strong convergence to the minimum norm minimizer: a dynamic approach combining time scaling, averaging, and Tikhonov regularization}, https://arxiv.org/pdf/2211.10140.pdf (2022)

\bibitem{att-com1996} {\sc H. Attouch, R. Cominetti}, {\it A dynamical approach to convex minimization coupling approximation with the steepest descent method},  Journal of Differential Equations 128(2), 519-540 (1996)

\bibitem{AC} {\sc H. Attouch, M.-O. Czarnecki}, {\it Asymptotic Control and Stabilization of Nonlinear Oscillators with Non-isolated Equilibria}, J. Differential Equations 179, 278-310 (2002)


\bibitem{AL-nemkoz} {\sc H. Attouch, S. L\'aszl\'o}, {\it Convex optimization via inertial algorithms with vanishing Tikhonov regularization: fast convergence to the minimum norm solution}, https://arxiv.org/abs/2104.11987 (2021)

\bibitem{APmathpr} {\sc H. Attouch, J. Peypouquet}, {\it Convergence of inertial dynamics and proximal algorithms governed by maximally monotone operators}, Math. Program. 174, 391–432 (2019)

\bibitem{APR} {\sc H. Attouch, J. Peypouquet, P. Redont}, {\it A dynamical approach to an inertial forward-backward algorithm for convex minimization}, SIAM Journal on Optimization 24(1),
232–256 (2014)

\bibitem{BT} {\sc A. Beck, M. Teboulle}, {\it A fast iterative shrinkage-thresholding algorithm for linear inverse problems}, SIAM J. Img. Sci. 2(1), 183–202 (2009)

\bibitem{BCL} {\sc R. I. Bo\c t, E. R. Csetnek, S.C. L\'aszl\'o}, {\it Tikhonov regularization of a second order dynamical system with Hessian damping}, Math. Program. 189, 151–186 (2021)

\bibitem{BCL1} {\sc R.I. Bo\c t, E.R. Csetnek, S.C. L\'aszl\'o}, {\it An inertial forward-backward algorithm for the minimization of the sum of two nonconvex functions}, EURO Journal on Computational Optimization 4, 3-25 (2016)

\bibitem{BCLstr} {\sc R. I. Bo\c t, E. R. Csetnek, S.C. L\'aszl\'o}, {\it On the strong convergence of continuous Newton-like inertial dynamics with Tikhonov regularization for monotone inclusions}, Journal of Mathematical Analysis and Applications,
530(2),  (2024)


\bibitem{BGMS} {\sc R.I. Bo\c t, S.M. Grad, D. Meier, M. Staudigl}, {\it Inducing strong convergence of trajectories in dynamical systems associated to monotone inclusions with composite structure}, Adv. Nonlinear Anal. {10}, 450–476 (2021)


\bibitem{CD} {\sc A. Chambolle, C. Dossal}, {\it On the convergence of the iterates of the fast iterative shrinkage/thresholding algorithm}, Journal of Optimization Theory and Applications 166, 968–982 (2015)

\bibitem{CPS} {\sc R. Cominetti, J. Peypouquet, S. Sorin}, {\it Strong asymptotic convergence of evolution equations governed by maximal monotone operators with Tikhonov regularization},  J. Differential Equations 245,  3753-3763 (2008)


\bibitem{Gu1} {\sc O. G\"uler}, {\it On the convergence of the proximal point algorithm for convex optimization}, SIAM J. Control Optim. 29, 403–419 (1991)

\bibitem{Gu2} {\sc O. G\"uler}, {\it New proximal point algorithms for convex minimization}, SIAM Journal on Optimization 2(4), 649–664 (1992)


\bibitem{JM-Tikh}{\sc M.A. Jendoubi, R. May}, {\it On an asymptotically autonomous system with Tikhonov type
regularizing term}, Archiv der Mathematik 95 (4),  389-399 (2010)


\bibitem{JM} {\sc P.R. Johnstone,P, Moulin}, {\it Local and global convergence of a general inertial proximal splitting scheme for minimizing composite functions}, Comput Optim Appl 67, 259–292 (2017)



\bibitem{L-jde} {\sc S.C. L\'aszl\'o}, {\it On the strong convergence of the trajectories of a Tikhonov regularized second order dynamical system with asymptotically vanishing damping}, Journal of Differential Equations 362, 355-381 (2023)

\bibitem{L-amop} {\sc S.C. L\'aszl\'o}, {\it Solving convex optimization problems via a second order dynamical system with implicit Hessian damping and Tikhonov regularization}, doi: 10.13140/RG.2.2.15237.12005, (2024)

\bibitem{LP} {\sc D.A. Lorenz, T. Pock}, {\it An inertial forward-backward algorithm for monotone inclusions}, Journal of Mathematical Imaging and Vision 51, 311–325 (2015)

\bibitem{MM} {\sc P.E. Maing\'e, A. Moudafi}, {\it Convergence of new inertial proximal methods for DC programming} SIAM J. Optim. 19, 397–413 (2008)

\bibitem{MO}  {\sc A.Moudafi, M. Oliny}, {\it Convergence of a splitting inertial proximal method for monotone operators}, Journal of Computational and Applied Mathematics 155(2), 447-454 (2003)

\bibitem{Nest1}{\sc  Y. Nesterov}, {\it  A method of solving a convex programming problem with convergence rate
$O(1/k^2)$}, Soviet Math. Dokl.  27, 372-376 (1983)

\bibitem{opial} {\sc Z. Opial}, {\it Weak convergence of the sequence of successive approximations for nonexpansive mappings}, Bulletin of the American Mathematical Society 73(4), 591–597 (1967)



\bibitem{Tikh}{\sc A. N. Tikhonov}, Doklady Akademii Nauk SSSR 151 (1963) 501-504, (Translated in "Solution of incorrectly formulated problems and the regularization method", Soviet Mathematics 4 (1963) 1035-1038)


\bibitem{TA}{\sc A. N. Tikhonov,  V. Y. Arsenin}, {\it Solutions of Ill-Posed Problems}, Winston, New York, (1977)

\end{thebibliography}
\end{document}